\documentclass{amsart}

\newif\ifdraft\draftfalse
\newif\ifcite\citefalse
\newif\ifblow\blowtrue

\usepackage{amsfonts,amssymb, amsmath}
\usepackage{oldgerm}
\usepackage{amscd}
\usepackage{lscape}
\usepackage{mathdots}
\usepackage[hyperindex]{hyperref}  
\usepackage[alphabetic]{amsrefs}
\ifdraft\ifcite\usepackage{showkeys}\else\usepackage[notcite,notref]{showkeys}\fi\fi

\newtheorem{theorem}{Theorem}[section]
\newtheorem{proposition}[theorem]{Proposition}

\newtheorem{lemma}[theorem]{Lemma}

\newtheorem{corollary}[theorem]{Corollary}

\theoremstyle{remark}

\theoremstyle{definition}

\theoremstyle{remark}

\newtheorem{remark}[theorem]{Remark}

\newtheorem{rem}[equation]{Remark}

\numberwithin{equation}{section}

\def\bc{\begin{cases}}
\def\ec{\end{cases}}

\def\ol{\overline}

\def\a{\alpha}
\def\Bbb{\mathbb}

\def\ch{{\mathcal H}}

\def\bc{{\mathbb C}}

\def\br{{\mathbb R}}

\def\bz{{\mathbb Z}}

\def\er{\eqref}

\def\bz{\mathbb Z}

\def\br{\mathbb R}
\def\bc{\mathbb C}

\def\lp2{L_pH_{2p}}
\def\bean{\begin{eqnarray}}
\def\eean{\end{eqnarray}}
\def\bea{\begin{eqnarray*}}
\def\eea{\end{eqnarray*}}
\def\beq{\begin{equation}}
\def\eeq{\end{equation}}
\def\beq*{\begin{equation*}}
\def\eeq*{\end{equation*}}
\def\bal{\begin{align*}}
\def\eal{\end{align*}}
\def\baln{\begin{equation}}
\def\ealn{\end{equation}}

\def\beg{\begin{gather*}}
\def\eng{\end{gather*}}
\def\bqu{\begin{question}}
\def\equ{\end{question}}

\def\ban{\begin{proof}[Answer]}
\def\ean{\end{proof}}

\def\ra{\Rightarrow}

\def\p{\partial}

\def\g{\gamma}

\def\on{\operatorname}

\def\bqu{\begin{question}}
\def\equ{\end{question}}

\def\0110{\begin{matrix} 0 & 1\\1&0\end{matrix}}

\def\fg{\mathfrak{g}}
\def\fh{\mathfrak{h}}

\def\fk{\mathfrak{k}}
\def\fl{\mathfrak{l}}

\def\fn{\mathfrak{n}}
\def\fo{\mathfrak{o}}
\def\fp{\mathfrak{p}}

\def\fs{\mathfrak{s}}

\def\ban{\begin{proof}[Answer]}
\def\ean{\end{proof}}
\def\wt{\widetilde}

\def\ben{\begin{equation}}
\def\een{\end{equation}}

\def\la{\langle}
\def\ra{\rangle}

\def\j1{{(j+1)}}

\def\e{\epsilon}

\newcommand\bcr{\bc\backslash \br_{\leq 0}}
\numberwithin{equation}{section} \allowdisplaybreaks

\begin{document}

\title
[General Toda systems with singular sources]
{Total masses of solutions to general Toda systems with singular sources}
\author{Debabrata Karmakar}
\email{debabrata@tifrbng.res.in}
\address{Tata Institute of Fundamental Research, Centre For Applicable Mathematics, Post Bag No 6503, GKVK Post Office, Sharada Nagar, Chikkabommsandra, Bangalore 560065, India}

\author{Chang-Shou Lin}
\email{cslin@math.ntu.edu.tw}
\address{Department of Mathematics, Taida Institute of Mathematical Sciences, National Taiwan University, Taipei 106, Taiwan}

\author{Zhaohu Nie}
\email{zhaohu.nie@usu.edu}
\address{Department of Mathematics and Statistics, Utah State University, Logan, UT 84322-3900, USA}

\author{Juncheng Wei}
\email{wei@math.cuhk.edu.hk}
\address{Department of Mathematics, Chinese University of Hong Kong, Shatin, NT, Hong Kong}



\begin{abstract}
In this article we obtain total masses of solutions to the Toda system associated to a general simple Lie algebra with singular sources at the origin. The determination of such total masses is one of the important steps towards establishing the a priori bound for solutions to the mean field type of Toda system on compact surfaces.
The total mass is found to be related to  the longest element $\kappa$ in the Weyl group of the corresponding Lie algebra. This is the foundation to future work relating the local blowup masses (from analysis) with the Weyl group.

This work generalizes the previous works in \cite{LWY}, \cite{ALW} and \cite{BC} for Toda systems of types $A, G_2$ and $B, C$. 
However, a more Lie-theoretic method is needed here for the general case, and the method relies heavily on the DPW method, Drinfeld-Sokolov gauge and the $W$-invariants. The last crucial step for the total masses is obtained by applying the work of Kostant \cite{KosToda} on the one dimensional Toda lattice.
\end{abstract}

\maketitle

\section{Introduction}
Let $\mathfrak{g}$ be a complex simple Lie algebra and $(a_{ij})$ be its Cartan matrix of rank $n$. In this paper, we consider the following
open Toda system on $\br^2$ associated to $\mathfrak{g}$ with singular sources at the origin:
\begin{equation} \label{Toda}
\begin{cases}
\displaystyle
\Delta u_i + 4\sum_{j=1}^n a_{ij}e^{u_j} = 4 \pi \gamma_i \delta_0, &  \gamma_i > -1,  \\
\displaystyle
\int_{\mathbb{R}^2} e^{u_i} \ dx < \infty, &1 \leq i \leq n,
\end{cases}
\end{equation}
where $\delta_0$ is the Dirac measure at the origin. Here a solution means $u=(u_1, u_2, \cdots, u_n)$ with $u_i \in C^2({\mathbb{R}^2 \backslash \{0\}})$
satisfying the equation  \eqref{Toda} with zero on RHS on $ \mathbb{R}^2 \backslash \{0\}$ and 
\begin{equation*}
u_i(x) = 2\gamma_i \log |x| + O(1) \ \ \mbox{near} \ 0.
\end{equation*}
When the Lie algebra $\fg = A_1= \fs\fl_2$ whose Cartan matrix is $(2)$, the Toda system becomes the Liouville equation 
\begin{equation}\label{liou}
\Delta u + 8 e^u = 4\pi \gamma\delta_0, \quad \gamma>-1, \quad \int_{\br^2} e^{u} \,dx <\infty.
\end{equation}

The Toda system \er{Toda} and the Liouville equation \er{liou} arise in many physical and geometric problems. For example, in the Chern-Simons theory, the Liouville equation is related to the abelian gauge field theory, while the Toda system is related to nonabelian gauges (see \cites{Yang, Taran}). We expect the work in this paper to have applications in constructing non-topological solutions of the Chern-Simons theory. 
On the geometric side, the Liouville equation 
is related to conformal metrics on $S^2$ with conical singularities whose Gaussian curvature is 1. The Toda systems are related to holomorphic curves in 
projective spaces \cite{Doliwa} and the infinitesimal Pl\"ucker formulas \cite{GH}, and the periodic Toda systems are related to harmonic maps \cite{Guest}. 

The main purpose of this paper is to study the asymptotic behavior of $u_i(x)$ or equivalently, to compute the total mass of $u_i:$
\begin{equation}\label{mass}
\sigma_i(u) = \frac{4}{2 \pi} \int_{\mathbb{R}^2} e^{u_i} \ dx, \ \ \ \ 1 \leq i \leq n.
\end{equation} 

The calculation of the total mass $\sigma_i(u)$ plays an extremely important role in the proof of a priori bounds of solutions 
of the following mean field type of Toda system on a compact surface $M$: 
\begin{equation}\label{meanToda}
\Delta_g u_i + \sum_{j=1}^n a_{ij} \rho_j \left(\frac{h_j e^{u_j}}{\int_M h_j e^{u_j}} - \frac{1}{|M|}\right)
= 4\pi \sum_{j = 1}^N \gamma_{ij} \left(\delta_{p_j} - \frac{1}{|M|}\right).
\end{equation}
 One major issue of \eqref{meanToda} is to establish the a priori bound for all solutions. Suppose blow-up solutions do exist. 
 The crucial step for studying bubbling solutions is to calculate the local mass of such solutions near each blow-up point. This calculation of local masses is a highly nontrivial problem. The calculation 
of $\sigma_i(u)$ for an entire solution is the first step towards solving this problem. For the $A_n$ cases, the a priori bounds
have been established successfully by carrying out this strategy in  \cite{LWZ17}.
The other cases will be pursued in future works. 

The classification problem for the solutions to the Toda systems has a long history. For the Liouville equation \er{liou}, Chen and Li \cite{CL} classified their solutions without the singular source, and 
Prajapat and Tarantello \cite{PT} completed the classification with the singular source. For general $A_n = \fs\fl_{n+1}$ Toda systems, Jost and Wang \cite{JW} classified the solutions without singular sources, and Ye and two of the authors \cite{LWY} completed the classification with singular sources. This later work also invented the method of characterizing the solutions by a complex ODE involving the $W$-invariants of the Toda system. The work \cite{LWY} has also established  the nondegeneracy result for the corresponding linearized systems. The case of $G_2$ Toda system was treated in \cite{ALW}. In \cite{BC}, one of us generalized the classification to Toda systems of types $B$ and $C$ by treating them as reductions of type $A$ with symmetries and by applying the results from \cite{N1}. The complete classification and nondegeneracy of Toda system \eqref{Toda} for general simple
Lie algebras shall be pursued in another paper.

For a solution $u = (u_1, u_2 , \cdots, u_n)$ of \eqref{Toda}, define
\begin{equation} \label{defineUi}
U_i = \sum_{j = 1}^n a^{ij}u_j \ \ \mbox{and} \ \  \gamma^i = \sum_{j = 1}^n a^{ij} \gamma_j \ \ \mbox{for} \ \  1 \leq i \leq n, 
\end{equation}
where $(a^{ij})$ is the inverse of the Cartan matrix $(a_{ij}).$ Then the $U_i$ satisfy
\begin{equation}\label{toda2}
\begin{cases}
\displaystyle
U_{i, z\bar z} + \exp\Big(\sum_{j=1}^n a_{ij} U_j\Big) = \pi \gamma^i \delta_0,\\
\displaystyle
\int_{\br^2} e^{u_i}\, dx<\infty,
\end{cases}
\end{equation}

We remark that the formula $\Delta = 4\frac\p{\p z}\frac\p{\p{\bar z}}$ is responsible for the slightly unconventional coefficient 4 on the left of \er{Toda}, which furthermore causes the coefficient $4$ in our definition of the total mass in \er{mass}. 
This coefficient can be easily dealt with (see \cite{BC}*{Remark 1.4}). 

 It turns out that the asymptotic behavior of the solution $u$ is related to the 
longest element of the Weyl group of the Lie algebra $\mathfrak{g}.$ Let $\kappa$ denote the longest element of the Weyl group. Then 
$-\kappa$ has the property that {\it it permutes positive simple roots $\{\alpha_1, \alpha_2, \cdots, \alpha_n\},$
i.e., $-\kappa \alpha_i$ is one of $\{\alpha_1, \alpha_2, \cdots, \alpha_n\}$}.  
Throughout the paper, we use $\{\alpha_1, \alpha_2, \cdots, \alpha_n\}$ to denote a fixed choice of positive simple roots of $\mathfrak{g}$ and $e_{\alpha_i}, e_{-\alpha_i}$
are the corresponding suitably chosen root vectors. The following is 
the main theorem of this article.

\begin{theorem} \label{mainthm}
The solutions $U_i$ in \eqref{toda2} satisfy that
\begin{equation} \label{asympexpUi}
U_i(z) = 2(\gamma^i - \langle\omega_i - \kappa \omega_i,w_0\rangle)\log|z| + O(1), \ \ \mbox{as} \ z \rightarrow \infty, \ \mbox{and} 
\end{equation}
\begin{equation} \label{totalmass}
\sigma_i(u) = 2 \langle \omega_i - \kappa \omega_i, w_0 \rangle.
\end{equation}
Moreover, $\sigma_i(u)$ is an even integer if $\gamma_i \in \mathbb{Z}_{\geq 0}.$
\end{theorem}
In Theorem \ref{mainthm},  $\langle \cdot , \cdot \rangle$
is the natural pairing between the real Cartan subalgebra $\mathfrak{h}_0$ and its dual $\mathfrak{h}_0^{\prime} := \mbox{Hom}(\mathfrak{h}_0, \mathbb{R}),$
$\omega_i \in \mathfrak{h}_0^{\prime}$ is the $i$-th fundamental weight and $w_0 \in \mathfrak{h}_0$ is defined by \eqref{definew0}. It is interesting that the total mass
is related to the longest element of the Weyl group. Recall that the total mass always satisfies the Pohozaev identity.
Theorem \ref{mainthm} shows that the RHS of \eqref{totalmass} satisfies the Pohozaev identity. Thus here is an example of the connection of the Pohozaev identity with the action of Weyl group. See \cite{LWZ17} in this aspect. This connection will be pursued further in future works.

Recall that $-\kappa \alpha_i \in \{\alpha_1, \alpha_2, \cdots, \alpha_n\}.$ If $-\kappa \alpha_i = \alpha_k,$ then the asymptotic behavior of $u_i$ has a clean expression.

\begin{theorem} \label{mainthmcor}
 If $-\kappa \alpha_i = \alpha_k$ and  $u = (u_1,u_2, \cdots, u_n)$ is a solution of \eqref{Toda} then
\begin{equation} \label{asympexpui}
u_i = -2(2 + \gamma_k) \log |z| + O(1), \ \ \mbox{as} \ z \rightarrow \infty.
\end{equation}
\end{theorem}

For the sake of  reference, we would like to write down all the total masses of solutions in terms of the singular data. For the case of $A_n, B_n, C_n$ and $G_2$, 
they have been given in \cite{LWY, ALW, BC}. So we list the remaining cases in the following statement.
\begin{corollary} \label{maincor}
The total masses of solutions are quantized as follows:
\begin{enumerate}
\item When $\mathfrak{g} = D_n,$ 
\begin{equation*} 
 \sigma_i(u)= 
\begin{cases}
4 \sum_{j = 1}^n a^{ij} (1 + \gamma_j), & \mbox{if $n$ is even}, \notag \\
4 \sum_{j = 1}^{n-2} a^{ij} (1 + \gamma_j) + 2 (a^{i, n-1} + a^{i n}) (2 + \gamma_{n-1} + \gamma_{n}), & \mbox{if $n$ is odd}.
\end{cases}
\end{equation*}
\item  When $\mathfrak{g} = E_6,$
\begin{align*} 
\sigma_i(u) &=
  2(a^{i1} + a^{i6})(2 + \gamma_1 + \gamma_6)
+ 2(a^{i3} + a^{i5})(2 + \gamma_3 + \gamma_5)  \notag \\
 &+ 4 a^{i2}(1 + \gamma_2) + 4 a^{i4}(1 + \gamma_4) \notag
\end{align*}
for all $1 \leq i \leq 6.$
\item When $\mathfrak{g} = F_4, E_7$ or $E_8,$
\begin{equation*}
\sigma_i(u)= 4 \sum_{j = 1}^n a^{ij} (1 + \gamma_j) \notag
\end{equation*}
for all $1 \leq i \leq n.$
\end{enumerate}
\end{corollary}
\begin{remark}
The entries of the inverse of the Cartan matrix of a simple Lie algebra are all positive \cite{InvC}. In Appendix B we have listed all the inverses of the Cartan matrices for simple Lie algebras.
\end{remark}

It is a well known fact that equation \eqref{toda2} is an integrable system \cite{LSbook}. From this point of view, Theorem \ref{mainthm} 
connects analysis of nonlinear PDEs with the theory of integrable systems. Our proof of Theorem \ref{mainthm} takes two advantages from the integrable system: the standard DPW method \cite{DPW} and the Drinfeld-Sokolov gauge \cite{DS}. The DPW method is the step to explain why \eqref{toda2}  
is an integrable system, that is, solutions can be obtained from some holomorphic data. The DPW method has been discovered in 
\cite{DPW} and developed into a very powerful method for some integrable systems related to nonlinear elliptic equations. We refer the reader to \cite{DPW} and \cite{GLI} and the references therein for the DPW method and its applications. However, our case is simpler
because the open Toda system only requires the standard decomposition of Lie groups, not loop groups. In any case, the DPW method could give us a local solution. On the other hand the well known Drinfeld-Sokolov gauge \cite{DS} yields the so called $W$-invariants, which are differential polynomials of the solutions (see section 4). More important is that {\it these $W$-invariants are rational functions with poles only at the origin.} For the case of $A_n,B_n,C_n$ and $G_2$, the complete set of $W$-invariants gives rise to a Fuchsian differential operator 
of one complex variable \cite{Feher1}. However, this does not hold  for other simple Lie algebras. These $W$-invariants are still essential in our approach. 
Indeed, applying $W$-invariants, our solutions can globally be expressed in terms of representations of the Lie algebra. 
To write down the statement we need to introduce the following notations. 

Let $G$ be a connected complex Lie group whose Lie algebra is $\mathfrak{g}.$ The classical Iwasawa decomposition says that $G=KAN$ (see Appendix subsection \ref{Iwasawasec}) with $K$ maximally compact, $A$ abelian and $N$ nilpotent. It is well known that the groups $A$ and $N$ are {\it simply connected.} 

Let $V_i$ be the $i$-th fundamental representation of $\mathfrak{g}$ and denote the highest weight vector by $|i \rangle.$ Let $V_i^*$ be the dual right representation of $\mathfrak{g}$ and choose the lowest weight vector denoted by by $\langle i |$ so that
$\langle i||i\rangle = 1.$
A fact in representation theory is that all the representations of $\mathfrak{g}$ can be lifted to $G$ if $G$ is simply connected. In particular, we can lift all the fundamental representations to $A$ and $N.$ Therefore the pairing $\langle i| g | i \rangle, g \in AN$ is well defined.

\begin{theorem} \label{main}
Let $\Phi : \mathbb{C} \backslash \mathbb{R}_{\leq 0} \rightarrow N$ be the unique solution of
\begin{equation}\label{develop0}
\begin{cases}
\displaystyle
\Phi^{-1}(z) \Phi_z(z) = \sum_{i=1}^n z^{\gamma_i} e_{-\alpha_i} \ \ \mbox{on} \ \mathbb{C} \backslash \mathbb{R}_{\leq 0} \\
\displaystyle
\lim_{z \rightarrow 0} \Phi(z) = Id,
\end{cases}
\end{equation}
where $Id \in G$ is the identity element. Then all the solutions to  \eqref{toda2} are
\begin{equation} \label{Ui}
U_i(z) = - \log \langle i| \Phi^*g \Phi | i\rangle + 2 \gamma^i \log |z|, \ \ 1 \leq i \leq n,
\end{equation}
where $g$ is of the form $g = C^* \Lambda^2 C, C \in N$ and $\Lambda \in A$, and ${}^*$ is defined in \er{dual}. 
\end{theorem}

However, our main concern is the asymptotic behavior of $U_i$ at $\infty.$ It seems that the expression \eqref{Ui} is not explicit enough to serve our purpose. Inspired by Kostant's work \cite{Kos Toda}, we find that the suitable expression can be written as an infinite series in the universal enveloping algebra of $\mathfrak{n}_-$ (i.e., an element of $\hat D (N)$ introduced in \cite{Kos Toda}). Using this expression, we could obtain the asymptotic of the R.H.S of \eqref{Ui} to prove Theorem \ref{mainthm}.


\vspace{0.2 cm}

The paper is organized as follows. Because of its importance to the analysis of Toda systems, we  first prove Theorem \ref{mainthm} in section 2 assuming Theorem \ref{main}. A crucial coefficient is assured nonzero 
by \cite{Kos Toda} (see \eqref{Kosresult1},\eqref{Kosresult2}). In section 3, we prove that any solution of \eqref{toda2}  locally can be expressed by some holomorphic data, and this process also explains why \eqref{toda2} is called an integrable system (other methods can be found in \cite{CWAnn}, \cite{ Uhlenbeck }, \cite{LSbook}). Together with $W$-invariants obtained by Drinfeld-Sokolov gauge, we prove Theorem \ref{main}
in sections 4, 5, and 6. The $W$-invariants play the major role in the process from the local expression to the global one. It is expected that the $W$-invariants has more applications to studying \eqref{toda2} with many singular points. In section 4 we introduce the $W$-invariants for equation \eqref{toda2} in a general domain
and obtain some explicit form for the case with single singularity in section 5. Finally, Theorem \ref{main} is proved in section 6. We give two Appendices. Among other things, Appendix A  gives very brief
explanation of some basic terminologies from representation theory of Lie algebra and Lie group, which is needed in section 3 and all subsequent sections. In Appnedix B we list all the Cartan matrices for simple Lie algebras and their inverse matrices.

\medskip
\noindent{\bf Acknowledgment.} Z. Nie thanks the University of British Columbia and the National Taiwan University for hospitality during his visits, where part of this work was done.  He acknowledges the Simons Foundation through Grant \#430297. 
The research of J. Wei is partially supported by NSERC of Canada. 
We thank Prof. L. Feh\'er for very useful correspondences on section 4.

\section{Proof of Theorem \ref{mainthm}}
In this section we will prove Theorem \ref{mainthm} and Theorem \ref{mainthmcor} by assuming Theorem \ref{main}.

Let us first introduce our basic setup in Lie algebras and representation theory. More details on representation theory can be found in Appendix A. Let $\mathfrak{g}$ be the complex simple Lie algebra for our Toda system \er{Toda}. Let $\mathfrak{h}$ be a fixed Cartan subalgebra, whose dimension $n$ is the rank of $\mathfrak{g}.$ Let $\mathfrak{g} = \mathfrak{h} \oplus \oplus_{\alpha \in \Delta} \mathfrak{g}_{\alpha}$ be the root space decomposition of $\mathfrak{g}$ with respect to $\mathfrak{h},$ where $\Delta$ denotes the set of roots. The roots are linear functionals on the Cartan subalgebra $\mathfrak{h},$ and for $X_{\alpha} \in \mathfrak{g}_{\alpha}$ and $H \in \mathfrak{h},$
we have $[H, X_{\alpha}] = \alpha(H)X_{\alpha}.$ 

Let $\Delta = \Delta^{+} \cup \Delta^{-}$ be a fixed decomposition of the set of roots into the sets of positive and negative roots,
and let $\Pi = \{\alpha_1, \alpha_2, \cdots, \alpha_n\}$ be the set of positive simple roots.

It is known that $\dim_{\mathbb{C}} \mathfrak{g}_{\alpha} = 1$ for all $\alpha \in \Delta,$ and we choose a Chevalley basis 
$\{e_{\alpha} \in \mathfrak{g}_{\alpha}, \alpha \in \Delta; h_{\alpha_i} \in \mathfrak{h}, 1 \leq i \leq n\}$ (see \cite{Humphrey}*{Theorem 25.2}). Then we have 
\begin{equation*}
[e_{\alpha_i}, e_{-\alpha_j}] = \delta_{ij}h_{\alpha_i}, \ \ \alpha_i(h_{\alpha_i}) = 2.
\end{equation*}
Furthermore, the Cartan matrix $(a_{ij})$ of $\mathfrak{g}$ is defined by 
\begin{equation}\label{cartan}
a_{ij} = \alpha_i(h_{\alpha_j}), \ \ 1 \leq i,j \leq n.
\end{equation}

Note that the real Cartan subalgebra is 
\begin{equation*}
\fh_0 = \{H\in \fh \,|\, \a(H)\in \br, \forall \a\in \Delta\} = \oplus_{i=1}^n \br h_{\a_i}.
\end{equation*}
Let $\mathfrak{h}_0^{\prime} := \mbox{Hom}(\mathfrak{h},\mathbb{R})$ denote the dual space of $\mathfrak{h}_0$ and 
let $\langle \cdot,\cdot\rangle$ denote the pairing of $\mathfrak{h}_0^{\prime}$ and $\mathfrak{h}_0.$ 

We introduce the following subalgebras of $\mathfrak{g}$
\begin{equation}\label{nn}
\mathfrak{n}_- = \oplus_{\alpha \in \Delta^+} \mathfrak{g}_{-\alpha}, \ \ \mathfrak{n}_+ = \oplus_{\alpha \in \Delta^+} \mathfrak{g}_{\alpha}.
\end{equation}
We also introduce the following subspace (from the principal grading \er{grading} of $\mathfrak{g}$)
\begin{equation}\label{g-1}
\mathfrak{g}_{-1} = \oplus_{i=1}^{n} \mathfrak{g}_{-\alpha_i}.
\end{equation}

This section is divided into $3$ subsections. In the first two subsections we introduce some notations and definitions and develop the necessary ingredients, and finally in the last subsection we give the proof of our main theorems. 







\subsection{Notations}
Following \cite{LWY}, we denote
\begin{equation*}
\mu_i = \gamma_i + 1 >0, \ \ \ 1\leq i \leq n.
\end{equation*}
Inspired by  \cite{KosToda} we introduce the following element
\begin{equation} \label{definew0}
 w_0 = \sum_{i = 1}^n \mu_i E_i \in \mathfrak{h}_0,
\end{equation}
 where $E_i = \sum_{k=1}^n a^{ki}h_{\alpha_k}.$ The elements $E_j$ and $w_0$ satisfy the properties
\begin{equation} \label{propertyejw0}
\alpha_i(E_j)= \langle\alpha_i, E_j \rangle = \delta_{ij} \quad \mbox{and} \quad 
\alpha_i(w_0)= \langle \alpha_i, w_0 \rangle = \mu_i.
\end{equation}
 We also introduce
\begin{align} \label{definexi}
\begin{split}
\zeta(z) &= \sum_{i = 1}^n z^{\gamma_i}e_{-\alpha_i} \in \mathfrak{g}_{-1} \\
\xi(z) &= z\zeta (z)= \sum_{i = 1}^n z^{\mu_i}e_{-\alpha_i} \in \mathfrak{g}_{-1},
\end{split}
\end{align}
where $z\in \bc\backslash \br_{\leq 0}$ and we use the principal branch for the power functions.  
To find the explicit expression for $\Phi$ in \eqref{develop0}, we introduce the following setup after \cite{KosToda}.

Let $\mathcal{S}$ be the set of all finite sequences
\begin{equation} \label{definitionofs}
s = (i_1,i_2, \cdots, i_k), \quad 1 \leq i_j \leq n,\ k \in \mathbb{N}.
\end{equation}
We denote by $|s|$ the length $k$ of the element $s \in \mathcal{S}.$ For $s \in \mathcal{S},$ we introduce 
\begin{equation} \label{definephi}
\begin{split}
\varphi(s) &= \sum_{j = 1}^{|s|} \alpha_{i_j}\in \fh_0',\\
\varphi(s,w_0) &= \varphi(s)(w_0) = \langle \varphi(s), w_0\rangle = \sum_{j = 1}^{|s|} \mu_{i_j}>0.
\end{split}
\end{equation}
Note that in the last equality 
we have used the property \eqref{propertyejw0}.
For $|s| =0,$ we define $\varphi(s) = 0.$

For $0 \leq j \leq |s| - 1,$ let $s_j \in \mathcal{S}$ be the sequence obtained from $s$ by ``cutting off" the first $j$ terms (different from
\cite{KosToda})
\begin{equation} \label{definesj}
s_j = (i_{j + 1}, \cdots, i_{|s|}),
\end{equation}
and define 
\begin{equation*}
\displaystyle{p(s,w_0) = \prod_{j = 0}^{|s| - 1} \langle \varphi(s_j), w_0\rangle =(\mu_{i_{1}} + \cdots +\mu_{i_{k}})\cdots(\mu_{i_{k-1}} + \mu_{i_{k}}) \mu_{i_{k}}.}
\end{equation*}
We also define $p(s, w_0) = 1$ when $|s| = 0.$ Clearly when $|s| \geq 1,$ we have
\begin{equation}\label{cancel}
p(s,w_0) = \varphi(s,w_0)\, p(s_1,w_0).
\end{equation}
Let $U(\mathfrak{n}_{-})$ be the universal enveloping algebra of $\mathfrak{n}_{-}$ introduced in \er{nn}. For convenience, write
$e_{-i} = e_{-\alpha_i}$ for $1\leq i\leq n.$ For $s \in \mathcal{S}$ as in \eqref{definitionofs}, define
\begin{equation} \label{definee_s}
e_{-s} = e_{-i_k} \cdots e_{-i_2}e_{-i_1}.
\end{equation}
We note that the $\xi$ in \eqref{definexi} is
\begin{equation} \label{expressionxi}
\xi(z) = \sum_{i=1}^n z^{\varphi((i), w_0)} e_{-(i)},
\end{equation}
where the $s$ are the simplest $(i)$ for $1 \leq i \leq n.$



\subsection{Supplementary lemmas}
Now we are going to prove a formula which will be used later. 
Recall that elements $e_{-s}$ as defined in \eqref{definee_s} belong to $U(\fn_-).$ 
Let $H \in \mathfrak{h},$ then we claim that  $[H, e_{-s}] = -\la \varphi(s),H\ra e_{-s}.$  Indeed,
\begin{equation} \label{computation}
\begin{split}
[H, e_{-s}] &= \sum_{j = 1}^{|s|} e_{-i_{|s|}} e_{-i_{|s| -1}} \cdots e_{-i_{j+1}}[H, e_{-i_j}] e_{-i_{j-1}} \cdots e_{-i_{1}}, \\
                  &=  \sum_{j = 1}^{|s|} e_{-i_{|s|}} e_{-i_{|s| -1}} \cdots e_{-i_{j+1}}(-\alpha_{i_j}(H) e_{-i_j}) e_{-i_{j-1}} \cdots e_{-i_{1}}, \\
                  &= \left(-\sum_{j = 1}^{|s|}\alpha_{i_j}(H)\right) e_{-i_{|s|}} e_{-i_{|s| -1}} \cdots e_{-i_{j+1}} e_{-i_j} e_{-i_{j-1}} \cdots e_{-i_{1}}, \\
                  &= -\la \varphi(s),H\ra e_{-s}.
\end{split}
\end{equation}

\begin{lemma} \label{welldefinednessofY}
Let $V$ be a finite dimensional representation of $\mathfrak{g},$ then there exists a positive integer $k$ such that $e_{-s}v =0,$ 
for all $v \in V$ whenever $|s| \geq k.$
\end{lemma}
\begin{proof}
Since every finite dimensional representation of $\mathfrak{g}$ can be decomposed as a direct sum of irreducible representations,
without loss of generality we can assume $V$ is irreducible. 
We can further decompose $V$ as a direct sum of its weight spaces. Therefore, it is enough to prove the lemma for a weight vector. Let  $v^{\beta}$ be a weight vector corresponding to the weight $\beta.$ Then using \eqref{computation} we see that
\begin{align*}
H(e_{-s}v^{\beta}) &= [H, e_{-s}] v^{\beta} + e_{-s}H v^{\beta}, \\
&= -\la \varphi(s),H\ra e_{-s} v^{\beta} + e_{-s} \la \beta,H\ra v^{\beta} \\
& = \la -\varphi(s) + \beta,H\ra (e_{-s} v^{\beta}).
\end{align*}
If $|s| \neq |s^{\prime}|$ and $e_{-s}v^{\beta} \neq 0, e_{-s^{\prime}}v^{\beta} \neq 0 $ then $-\varphi(s) + \beta$
and $-\varphi(s^{\prime}) + \beta$ are different weights because $\varphi(s)$ and $\varphi(s^{\prime})$ are different as elements 
in $\mathfrak{h}_0'.$ Since there are only finitely many weights we have $e_{-s}v^{\beta} = 0,$ if $|s|$ is large. This
completes the proof of the lemma.
\end{proof}
Define formally, the operator
\begin{equation} \label{YZ}
\mathcal{Y}(z) = \sum_{s \in \mathcal{S}} \frac{z^{\varphi(s,w_0)}}{p(s,w_0)}e_{-s}, \ \ \ \mbox{for} \ z \in \mathbb{C} \backslash
\mathbb{R}_{\leq 0}.
\end{equation}
Define the action of $\mathcal{Y}(z)$ on a finite dimensional representation space $V$ by
\begin{equation} \label{YZaction}
\mathcal{Y}(z) v= \sum_{s \in \mathcal{S}} \frac{z^{\varphi(s,w_0)}}{p(s,w_0)}(e_{-s}v), \ \ \ \mbox{for} \ v \in V,  \ z \in \mathbb{C} \backslash
\mathbb{R}_{\leq 0}.
\end{equation}
By Lemma \ref{welldefinednessofY}, we see that the sum in \eqref{YZaction} is a finite sum and hence 
$\mathcal{Y}(z)v$ is well defined for all $v \in V$ and  $z \in \mathbb{C} \backslash
\mathbb{R}_{\leq 0}.$ The next proposition shows that this $\mathcal{Y}(z)$ is identical with $\Phi(z)$ as an operator acting on $V$. 

\begin{proposition} \label{actionsame}
Let $V$ be a finite dimensional representation of $\mathfrak{g}.$
Then the action of $\Phi(z) $ on $V$
coincides with the action of $\mathcal{Y}(z)$ on $V,$ for all $ \ z \in \mathbb{C} \backslash
\mathbb{R}_{\leq 0},$ where $\Phi$ is defined as in \eqref{develop0}.
\end{proposition}

\begin{proof}
We will show that the right action of $\Phi(z)$ and $\mathcal{Y}(z)$ on $V^{*}$ are the same, and then the proposition 
follows because
\begin{align*}
\langle w, \Phi(z)v \rangle = \langle w \Phi(z), v \rangle = \langle w\mathcal{Y}(z),v \rangle = \langle w,\mathcal{Y}(z)v \rangle
\end{align*}
for all $ v \in V, w \in V^{*}$ and  $z \in \mathbb{C} \backslash \mathbb{R}_{\leq 0}.$

Now we show that $w\Phi(z) = w\mathcal{Y}(z)$ for all $w \in V^{*},  \ z \in \mathbb{C} \backslash
\mathbb{R}_{\leq 0}.$ From \eqref{develop0} it follows that $w \Phi :  \mathbb{C} \backslash
\mathbb{R}_{\leq 0} \rightarrow V^{*} $ satisfies the ODE
\begin{align*}
\partial_z (w \Phi(z)) &= (w \Phi(z)) \zeta, \ \ \mbox{for} \  z \in \mathbb{C} \backslash \mathbb{R}_{\leq 0}, \\
w\Phi(0) &= w.
\end{align*}
Therefore it is enough to show that $w\mathcal{Y}(z)$ satisfies the same ODE with the same initial condition.

Since $\varphi(s, w_0)>0,$ the initial value $w\mathcal{Y}(0)$ equals to $w.$
Moreover, 
\begin{equation}\label{ode for y} 
\begin{split}
\partial_z (w \mathcal{Y}(z)) &= \sum_{s \in \mathcal{S}} \frac{\varphi(s,w_0) z^{\varphi(s,w_0)-1}}{p(s,w_0)}(we_{-s}), \\
&=  \frac{1}{z}\sum_{|s| \geq 1} \frac{ z^{\varphi(s,w_0)}}{p(s_1,w_0)}(we_{-s})
\end{split}
\end{equation}
by \er{cancel}. 
For $s \in \mathcal{S}$ and $1\leq i\leq n$, define $t^i = (i,s).$ Then from \eqref{definephi} and \eqref{expressionxi} we have
\begin{align} \label{expression1}
 \frac{ z^{\varphi(s,w_0)}}{p(s,w_0)}(we_{-s}) \xi &= \sum_{i=1}^n  \frac{ z^{\varphi(t^i,w_0)}}{p(s,w_0)}(we_{-t^i}), \notag \\
&=  \sum_{i=1}^n  \frac{ z^{\varphi(t^i,w_0)}}{p((t^i)_1,w_0)}(we_{-t^i}).\notag
\end{align}
Therefore summing over all $s \in \mathcal{S}$ we get 
\begin{equation*}
(w \mathcal{Y}(z)) \xi = \sum_{|s| \geq 1} \frac{ z^{\varphi(s,w_0)}}{p(s_1,w_0)}(we_{-s}).
\end{equation*}
Therefore \er{ode for y} implies that $\partial_z(w \mathcal{Y}(z)) = \frac{1}{z}(w \mathcal{Y}(z))\xi = (w\mathcal{Y}(z))\zeta,$ and hence completes the proof.
\end{proof}

\begin{remark}
It can be shown that the $\Phi$ defined by a system of ODE from \eqref{develop0} is also uniquely characterized by 
the following algebraic condition inspired by \cite{Kos Toda}:
\begin{equation*}
\begin{cases}
\Phi^{-1} w_0 \Phi = w_0 - \xi, \\
\Phi(0) = Id.
\end{cases}
\end{equation*}
\end{remark}

Next we will use this explicit expression for $\Phi$ in Proposition \ref{actionsame} to prove our main Theorem \ref{mainthm}.

\subsection{Proofs of Theorems}
Let $\lambda$ be a dominant integral weight and let $V^{\lambda}$ be the corresponding irreducible representation. Let $\kappa$ be the longest Weyl group element which maps positive roots to negative roots. Then $\kappa \lambda$ {\it is the lowest weight of} $V^{\lambda}$. 
 In the following calculations, we use the Hermitian metric $\{\cdot,\cdot\}$ on $V^{\lambda}$ which is invariant under the compact subgroup $K^s$ of $G^s$ (see \er{adjoint}). It also follows from the construction of $\{\cdot, \cdot\}$ that {\it the weight vectors belonging to different weight spaces in $V^{\lambda}$ are orthogonal with respect to} $\{\cdot,\cdot\}.$ 
 Choose vectors $v^{\lambda}$ in $V_{\lambda}$  and 
$v^{\kappa \lambda}$ in $V_{\kappa \lambda}$ in the one dimensional highest and lowest weight spaces such that 
\begin{equation*}
\{v^{\lambda}, v^{\lambda} \} = 1, \ \  \ \{v^{\kappa \lambda},v^{\kappa \lambda}\} = 1.
\end{equation*}
(One can choose $v^{\kappa \lambda}$ to be $s_0(\kappa) v^{\lambda},$ where $s_0(\kappa) \in G^s$ induces
the longest Weyl element $\kappa.$ See \cite{KosToda}*{Eq. (5.2.10)}.)

With the notation \eqref{definephi} and \eqref{definee_s}, the vector $e_{-s}v^{\lambda}$ is a weight vector with weight 
$-\varphi(s) + \lambda.$ Since different weight spaces are orthogonal, by Proposition \ref{actionsame} 
we have
\begin{equation} \label{computation2}
\left\{\Phi v^{\lambda}, v^{\kappa \lambda}\right\} = 
\left(\sum_{s \in \mathcal{S}^{\lambda}} \frac{c_{s,\lambda}}{p(s,w_0)}\right) z^{\langle\lambda - \kappa \lambda, w_0\rangle},
\end{equation}
where $\mathcal{S}^{\lambda} = \{s \in \mathcal{S} | -\varphi(s) + \lambda = \kappa \lambda\},$ and
$c_{s, \lambda} = \{e_{-s}v^{\lambda}, v^{\kappa \lambda} \}.$ 

\bigskip

\noindent{\bf Proof of Theorem \ref{mainthm} and Corollary \ref{maincor}.}
\begin{proof}
To show the asymptotic behavior of  \eqref{asympexpUi} we prove 
\begin{equation} \label{asymptbehavior1}
e^{-U_i} = |z|^{-2\gamma^{i}} |z|^{2\langle \omega_i - \kappa \omega_i, w_0 \rangle} \left(c_i + o(1)\right),
\ \ \mbox{as} \ z \rightarrow \infty,
\end{equation}
where $c_i >0.$
From \eqref{Ui} and \eqref{use vi}, we get
\begin{equation*}
e^{-U_i} = |z|^{-2\gamma^i} \langle i| \Phi^{*}C^{*}\Lambda^2 C \Phi|i \rangle
= |z|^{-2\gamma^i} \left\{ \Lambda C \Phi v^{\omega_i}, \Lambda C \Phi v^{\omega_i}\right\}.
\end{equation*}
Now we know that 
\begin{equation*}
 \left\{ \Lambda C \Phi v^{\omega_i}, \Lambda C \Phi v^{\omega_i}\right\}= \sum_{v}
\left|\left\{ \Lambda C \Phi v^{\omega_i}, v\right\}\right|^2,
\end{equation*}
where the sum ranges over an orthonormal basis of the $i$-th fundamental representation $V_i$ consisting of weight vectors.  

Let $v^{\beta}$ be a weight vector corresponding to a weight $\beta.$ Then from \eqref{adjoint} we get
\begin{equation*}
\left\{ \Lambda C \Phi v^{\omega_i}, v^{\beta}\right\} = \left\{\Phi v^{\omega_i}, C^{*}\Lambda v^{\beta}\right\}
= e^{\beta(H)} \left\{\Phi  v^{\omega_i}, C^{*}v^{\beta}\right\},
\end{equation*}
where $\Lambda = \exp(H)$ with $H \in \mathfrak{h}_0.$ Since $C^{*} \in {N}_{+},$ the subgroup of $G$ corresponding to $\fn_+$, we can write
$C^{*}v^{\beta} = v^{\beta} + w^{\beta},$ where $w^{\beta}$ belongs to the direct sum of the weight spaces
with weights strictly larger that $\beta$, that is, $w^\beta \in \oplus_{\a\in \Delta^+} V_{i, \beta + \a}$.  
Therefore, we have
\begin{equation*}
\left\{ \Lambda C \Phi v^{\omega_i}, v^{\beta}\right\}
= e^{\beta(H)}\left[\left\{\Phi v^{\omega_i}, v^{\beta}\right\} + \left\{\Phi v^{\omega_i}, w^{\beta}\right\}\right].
\end{equation*}
Now, if $e_{-s}v^{\omega_i} \neq 0,$ then it is a weight vector with weight $-\varphi(s) + \omega_i$
and since two weight vectors belonging to two different weight spaces are orthogonal, using Proposition
\ref{actionsame}, we see that the highest power of $z$ in $\left\{ \Lambda C \Phi v^{\omega_i}, v^{\beta}\right\}
$ is equal to the power of $z$ in $\left\{\Phi v^{\omega_i}, v^{\beta}\right\},$ which is $z^{\la\omega_i-\beta, w_0\ra}$, provided its coefficient is nonzero. To prove this we write $w^{\beta} = \sum_{\alpha \in \Delta^+} c_{\alpha} v^{\beta + \alpha},$ where $v^{\beta + \alpha} \in V_{\beta + \alpha}$ and $c_{\alpha} \in \mathbb{C}.$ Then using Proposition \ref{actionsame} we see that 
\begin{align*}
\{\Phi v^{\omega_i}, v^{\beta}\} = \left(\sum_{-\varphi(s) + \omega_i = \beta}\frac{\{e_{-s}v^{\omega_i},v^{\beta}\}}{p(s,w_0)}\right)z^{\langle \omega_i - \beta, w_0\rangle},
\end{align*}
 and 
\begin{align*}
\{\Phi v^{\omega_i}, w^{\beta}\} =\sum_{\alpha \in \Delta^+} c_{\alpha}\left(\sum_{-\varphi(s) + \omega_i = \beta + \alpha}\frac{\{e_{-s}v^{\omega_i},v^{\beta + \alpha}\}}{p(s,w_0)}\right)z^{\langle \omega_i - \beta - \alpha, w_0\rangle},
\end{align*}
Since $\langle \alpha, w_0 \rangle > 0,$ for any $\alpha\in \Delta^+$ (by \eqref{propertyejw0}) the conclusion follows. Thus we have proved that for $v^{\beta} \in V_{\beta}$ the highest power of $z$ in $\{ \Lambda C \Phi v^{\omega_i}, v^{\beta}\}$ is at most $z^{\langle \omega_i - \beta, w_0 \rangle }.$ Now $\kappa\omega_i$ is the lowest weight and all other weights $\beta$ are of the form $\beta = \kappa \omega_i + \sum_{j = 1}^n n_j \alpha_j,$ with $n_j \in \mathbb{N} \cup \{0\}$ and at least one $n_j >0.$ Therefore, $\langle \omega_i - \beta, w_0 \rangle < \langle \omega_i - \kappa\omega_i, w_0 \rangle $ for all weights $\beta \neq \kappa\omega_i,$
and hence  the leading term in $\sum_{v} \left|\left\{ \Lambda C \Phi v^{\omega_i}, v\right\}\right|^2$ corresponds to the leading term in $\left|\left\{ \Lambda C \Phi v^{\omega_i}, v^{\kappa \omega_i}\right\}\right|^2.$
From \eqref{computation2}, we have
\begin{equation*}
\left\{\Phi v^{\omega_i}, v^{\kappa \omega_i}\right\} = 
\left(\sum_{s \in \mathcal{S}^{\omega_i}} \frac{c_{s,\omega_i}}{p(s,w_0)}\right) z^{\langle \omega_i - \kappa\omega_i, w_0\rangle}.
\end{equation*}
The coefficient $\sum_{s \in \mathcal{S}^{\omega_i}} \frac{c_{s,\omega_i}}{p(s,w_0)}$ is considered
in \cite{Kos Toda}*{Eq (5.95)}. (Note our different convention in \eqref{definesj} for $s_j,$ so our $s$ corresponds
to the $\bar s$ there). It is shown in \cite[Proposition 5.5.1]{Kos Toda} that there exists an element $d(w)$ depending on $w_0$ such that $d(w) \in \mathcal{H}_0$, the real Cartan subgroup of $G$ corresponding to $\fh_0$, and in \cite[Proposition 5.9.1]{Kos Toda} that
\begin{equation} \label{Kosresult1}
\sum_{s \in \mathcal{S}^{\omega_i}} \frac{c_{s,\omega_i}}{p(s,w_0)} = \pm d(w)^{- \kappa \omega_i}.
\end{equation}
Here again if $d(w) = \exp(H_1)$ with $H_1 \in \mathfrak{h}_0,$ $d(w)^{-\kappa \omega_i}$ is defined by
\begin{equation} \label{Kosresult2}
d(w)^{-\kappa \omega_i} = e^{-\kappa \omega_i(H_1)}>0.
\end{equation}
Therefore,
\begin{equation*}
\left\{ \Lambda C \Phi v^{\omega_i}, \Lambda C \Phi v^{\omega_i}\right\}=
e^{\kappa \omega_i(H)} \left(\left\{\Phi v^{\omega_i}, v^{\kappa \omega_i}\right\} + \mbox{lower order terms} \right),
\end{equation*}
and hence \eqref{asymptbehavior1} holds with
\begin{equation} \label{defineci}
c_i = \left(e^{\kappa \omega_i(H)}\right)^2 \left(\sum_{s \in \mathcal{S}^{\omega_i}} \frac{c_{s,\omega_i}}{p(s,w_0)}\right)^2 >0.
\end{equation}
Now integrating \eqref{toda2} and using \eqref{asympexpUi}, we have 
\begin{align*}
4 \int_{\mathbb{R}^2} e^{u_i} \ dx &= 4\pi \gamma^i - \lim_{R \rightarrow \infty} \int_{\partial B_R}
\frac{\partial U_i}{\partial \nu} \ ds \\
&=4\pi \langle \omega_i - \kappa \omega_i, w_0 \rangle,
\end{align*}
and hence $\sigma_i(u) = 2 \langle \omega_i - \kappa \omega_i, w_0 \rangle.$  This proves \eqref{totalmass}.

By the definition of the Weyl group and that $\omega_i$ belongs to the dominant chamber of the integral weight lattice, it is known that $\omega_i - \kappa \omega_i$ belongs to the root lattice, that is, there exist $n_{ij}\in \bz_{\geq 0}$ such that 
$\omega_i - \kappa \omega_i = \sum_{j=1}^n n_{ij}\a_j.
$
Then if $\gamma_j\in \bz_{\geq 0}$, we have 
$$
\sigma_i(u) = 2\la \omega_i - \kappa \omega_i, w_0\ra = 2\sum_{j=1}^n n_{ij} (1+\gamma_j)
$$
is an even integer. 

For concreteness and to prove Corollary \ref{maincor}, we also proceed as follows. Using $\omega_i = \sum_{j=1}^n a^{ij}\alpha_j$ we get
\begin{equation} \label{computew}
2 \langle \omega_i - \kappa \omega_i, w_0\rangle = 2\sum_{j = 1}^n a^{ij} \langle \alpha_j - \kappa \alpha_j, w_0\rangle.
\end{equation} 
Since $-\kappa$ permutes the simple roots $\{\alpha_1, \alpha_2, \cdots, \alpha_n\}, -\kappa \alpha_i = \alpha_k$ for some $\alpha_k.$  
 In order to prove that $\sigma_i(u)$ is an even
integer when $\gamma_i \in \mathbb{Z}_{\geq 0}$ we need to compute the $-\kappa\alpha_i, 1 \leq i\leq n.$

Except for the three cases $A_n, E_6$ and $D_n$(with $n$ odd), $-\kappa = Id.$ 
Therefore using  the property  \eqref{propertyejw0} we get from \eqref{computew}
\begin{equation*}
\sigma_i(u) =4\sum_{j=1}^n a^{ij} \langle \alpha_j , w_0\rangle =4 \sum_{j=1}^n a^{ij}(1 + \gamma_j),
\end{equation*}
except for the Lie algebras $A_n, E_6$ and $D_n$(with $n$ odd). Now we see that except for the Lie algebras mentioned above, 
$2a^{ij} \in \mathbb{N}$ for all $1\leq i,j \leq n$ (see Appendix B). This proves that $\sigma_i(u) $ is an even integer if $\gamma_{i} \in \mathbb{Z}_{\geq 0}$ except for those three Lie algebras. 

For the $A_n$ case
\begin{equation*}
-\kappa \alpha_i = \alpha_{n+1-i}, \ \ 1\leq i \leq n,
\end{equation*}
so that we have
\begin{align*}
\sigma_i(u) = 2 \sum_{j = 1}^n a^{ij}(2 + \gamma_j + \gamma _{n+1-j}) =2 \sum_{j=1}^n (a^{ij} + a^{i, n+1-j})(1 + \gamma_j), \ 1 \leq i \leq n.
\end{align*}
The inverse Cartan matrix for $A_n$ satisfies $a^{ij} + a^{i, n+1-j} \in \mathbb{N}$ (see Appendix B) and hence this proves  the result for the $A_n$ case.
 
For the $D_{n}$ case with $n$ odd, we have
\begin{align*}
-\kappa \alpha_{n-1} = \alpha_{n},\quad -\kappa \alpha_{n} = \alpha_{n-1},\quad -\kappa \alpha_i = \alpha_i, \ 1 \leq i \leq n-2,
\end{align*}
that is, $-\kappa$ permutes the last two roots and preserves the others. This proves $(i)$ of Corollary \ref{maincor}. We can easily check that for $D_n$ (with $n$ odd), $a^{i, n-1} + a^{in}\in \mathbb{N}$ and $ 2a^{ij} \in \mathbb{N}$ for all $1 \leq i \leq n, 1 \leq j \leq n-2,$ and hence proves that $\sigma_i(u) \in 2 \mathbb{N}$ provided $\gamma_i \in \mathbb{Z}_{\geq 0}, 1\leq i \leq n.$

For $E_6$ we have
\begin{align*}
-\kappa \alpha_1 = \alpha_6, \ &-\kappa \alpha_6 = \alpha_1, \\
-\kappa \alpha_3 = \alpha_5, \ &-\kappa \alpha_5 = \alpha_3, \\
-\kappa \alpha_2 = \alpha_2, \ &-\kappa \alpha_4 = \alpha_4. 
\end{align*}
This then proves  $(ii)$ of Corollary \ref{maincor}.  The inverse Cartan matrix for $E_6$ satisfies $a^{i1}+ a^{i6}, a^{i3}+a^{i5}, a^{i2},
a^{i4} \in \mathbb{N},$ and hence $\sigma_i(u) \in 2 \mathbb{N}$ if $\gamma_i\in \mathbb{Z}_{\geq 0}.$

This completes the proof of Theorem \ref{mainthm} together with Corollary \ref{maincor}.
\end{proof}

\medskip

\noindent{\bf Proof of Theorem \ref{mainthmcor}.}
\begin{proof}
Since $u_i = \sum_{j=1}^n a_{ij}U_j, \gamma_i = \sum_{j=1}^n a_{ij}\gamma^j, \alpha_i = \sum_{j=1}^n a_{ij}\omega_j,$ and $\langle\alpha_i,w_0\rangle = \mu_i,$ and assuming that $-\kappa\alpha_i = \alpha_k$
for some $1\leq k\leq n,$ we see that 
\begin{align*}
u_i &= \sum_{j = 1}^n a_{ij}U_j = 2(\gamma_i - \langle \alpha_i - \kappa\alpha_i, w_0\rangle)\log |z| + O(1) \\
&= 2(\gamma_i - \mu_i - \mu_k) \log |z| + O(1), \\
&= -2(2+\gamma_k) \log |z| + O(1),\quad \text{as }z\to \infty. 
\end{align*}
This is \eqref{asympexpui} and completes the proof.
\end{proof}



\section{The DPW method and Local solutions}

In this section, we show that any solution to the Toda system \er{toda2} locally comes from holomorphic data. Our approach follows  \cite{GL}, \cite{LS}. 
In \cite{GL}, a similar process was applied for the periodic Toda systems by the Iwasawa decomposition for loop groups. 


For simplicity, we introduce the notation $\bc^* = \bc \backslash \{0\}$ and also recall $\fg_{-1}$ from \er{g-1} and 
the decomposition $\mathfrak{g} = \mathfrak{n}_- \bigoplus \mathfrak{n}_+ \bigoplus \mathfrak{h}$ in view of \er{nn}. 
 Before we proceed we would like to mention a few words about the Gauss decomposition of Lie groups, which is one main ingredient 
of our proof. 

Let $G$ be a connected complex Lie group whose Lie algebra is $\fg$. Let $\ch$ be the Cartan subgroup of $G$ corresponding to $\fh$. Denote the subgroups of $G$ corresponding to $\fn_-$ and  $\fn_+$ by $N_-$ and $N_+$ respectively.
The Gauss decomposition (see \cite{LSbook} and  \cite{KosToda}*{Eq. (2.4.4)}) says that there exists an open and dense subset $G_r$ of $G$, called the regular part, such that 
\begin{equation}\label{Gdec}
G_r = N_-\ch N_+ = N_- N_+ \ch.
\end{equation}
We note that $\ch N_+ = N_+ \ch$ because $hnh^{-1} \in N_+$, where $h\in \ch$ and $n\in N_+$. 
Clearly $G_r$ contains the identity element of $G$.
 Now if $L \in N_-$ then $L^* \in N_+$ (see \eqref{dual}). Since the groups $N_+$ and $N_-$ are simply connected  the quantity $\la i | L^* L | i \ra$ is well defined.
\begin{theorem}\label{Taiwan} Let $U = (U_1, U_2, \cdots, U_n)$ be a solution to \er{toda2} and $z_0 \in \mathbb{C}^*.$ Then the following results hold:
\begin{enumerate}
\item There exist holomorphic maps $\eta$ and $L$ from a neighborhood $D$ of $z_0$ to $\mathfrak{g}_{-1}$ and $N_{-}$
such that 
\begin{equation} \label{the eta}
L^{-1}(z)L_z(z) = \eta (z) :=\sum_{i=1}^n f_i(z) e_{-\alpha_i}, \ \ L(z_0) = Id,
\end{equation}
and $f_i(z) \neq 0$ in $D.$

\item Each component of $U$ can be written as:
\begin{equation}\label{loc sol}
U_i = -\log \la i | L^* L | i \ra + \sum_{j=1}^n a^{ij} \log |f_j|^2, \ z \in D.
\end{equation}
\end{enumerate}
\end{theorem}

\begin{proof}
With out loss of generality we can assume $z_0 = 1.$ Define
\begin{align} \label{Aaa}
A &= -\sum_{i=1}^n \tfrac{1}{2}U_{i, z} h_{\a_i} + \sum_{i=1}^n \exp\Big({\tfrac{1}{2} \sum_{j=1}^n a_{ij} U_j}\Big) e_{-\a_i},
\end{align}
then
\begin{align}  \label{Att}
A^\theta &= \sum_{i=1}^n \tfrac{1}{2}U_{i, \bar z} h_{\a_i} - \sum_{i=1}^n \exp\Big({\tfrac{1}{2} \sum_{j=1}^n a_{ij} U_j}\Big) e_{\a_i},
\end{align}
where $\theta$ is defined in \er{def theta}. 
Eq. \er{toda2} has the following zero-curvature equation on $\bc^*$
\begin{align}
[ \p_z + A, \p_{\bar z} + A^\theta] &= 0, \quad \text{that is,}\label{zero-curv}\\
-A_{\bar z} + (A^\theta)_z + [A, A^\theta] &= 0. \notag
\end{align}

The zero-curvature equation can also be written as the Maurer-Cartan equation 
$$
d\omega + \frac{1}{2}[\omega, \omega]=0
$$
for the following Lie algebra valued differential form 
$$
\omega = A\, dz + A^\theta d\bar z \in \Omega^1(\bc^*, \fg).
$$
With $dz = dx_1+ i\,dx_2$ and $d\bar z = dx_1 - i\,dx_2$, it is also 
$$
\omega = (A + A^\theta) dx_1 + i(A - A^\theta) dx_2. 
$$
Since the Cartan involution $\theta$ on $\fg$ is conjugate linear, we see that the zero curvature connection $\omega$ takes value in the fixed subalgebra $\fg^\theta = \fk$.
Therefore by \cite{Sharpe}*{Theorems 6.1 and 7.14}, there exists a map from the simply-connected domain
$$
F : \bcr \to K\subset G
$$
to the compact subgroup $K=G^\theta$ such that 
\begin{equation}
\label{get F}
\begin{cases} 
F^{-1} dF = \omega\\
F(1) = Id.
\end{cases}
\end{equation}
Therefore, 
\begin{equation}\label{more}
F^{-1} F_z = A,\quad F^{-1} F_{\bar z} = A^\theta.
\end{equation}

For a small neighbourhood $D$ of $1$, the map $F$ has the following Gauss decomposition 
\begin{equation}\label{decomp}
F  = L M \exp(H)
\end{equation}
where $L : D\to N_-$, $M: D\to N_+$, 
 $H = \sum_{i=1}^n b_i h_{\a_i} : D\to \fh$ and $\exp:\fh\to \ch$ is the exponential map to the Cartan subgroup. 
From $F(1)=Id$ in \er{get F}, we see clearly that $L(1)=Id$. 


Now we show that 
$L$ is holomorphic in $D$. By the second equation in \er{more}, we have
$$
\exp(-H)M^{-1}(L^{-1} L_{\bar z}) M\exp(H) +  \exp(-H)M^{-1} M_{\bar z} \exp(H) + H_{\bar z} = A^\theta.
$$
In view of \er{Att}, the components in $\fn_-=\oplus_{\a\in \Delta^+} \fg_{-\a}$ of the above equation give 
\begin{equation}\label{L hol}
L^{-1} L_{\bar z} = 0,
\end{equation}
and then the components in $\fh$ give
$$
b_{i,\bar z} = \tfrac{1}{2} U_{i, \bar z}, \quad 1\leq i\leq n.
$$
Thus we see that 
$
b_{i,z\bar z} = \tfrac{1}{2} U_{i, z\bar z}.
$
Taking the conjugate, we also have 
$
\bar b_{i,z\bar z} = \tfrac{1}{2} U_{i, z\bar z}
$
since $U_i$ is real. Therefore,
$$
(b_i + \bar b_i)_{z\bar z} = U_{i, z\bar z}.
$$
Hence we have, for $1\leq i\leq n$, 
\begin{equation}\label{bu}
b_i(z) + \bar b_i(z) = U_i(z) - p_i(z)
\end{equation}
for some real-valued harmonic function $p_i$ in $D$. 

By the first equation in \er{more}, we have 
\begin{equation}\label{not zero}
\exp(-H)M^{-1}(L^{-1} L_{z}) M\exp(H) +  \exp(-H)M^{-1} M_z \exp(H) + H_{z} = A.
\end{equation}
Since $A\in \fg_{-1} \oplus \fh$ by \er{Aaa}, we see that $L^{-1} L_z\in \fg_{-1}$. We denote it by $\eta$ and write it out in terms of the basis
\begin{equation}\label{Lcomb}
L^{-1} (z)L_z(z) = \eta(z) = \sum_{i=1}^n f_i(z) e_{-\a_i}.
\end{equation}
Then the $f_i$s are holomorphic by \er{L hol}. Furthermore, by \er{Aaa} the component of $A$ in $\fg_{-1}$ is $\sum_{i=1}^n \exp\Big({\tfrac{1}{2} \sum_{j=1}^n a_{ij} U_j}\Big) e_{-\a_i}$. 
We also see that 
\begin{multline*}
\text{the component of the LHS of }\er{not zero} \text{ in }\fg_{-1} = \exp(-H)(L^{-1}L_z)\exp(H) \\
= \exp(-H) \bigg(\sum_{i=1}^n f_i(z) e_{-\a_i}\bigg) \exp(H) = \sum_{i=1}^n f_i(z) e^{\a_i(H)} e_{-\a_i}.
\end{multline*}
Comparing the components in $\mathfrak{g}_{-1}$ from both sides of  \eqref{not zero} we see that the $f_i$ in \er{Lcomb} are nowhere zero in $D$. 
Thus we have shown  \er{the eta}.

Now following the physicists, we denote by $|i\ra$ a highest weight vector in the $i$th fundamental representation $V_i$ of $\mathfrak{g}$, and $\la i|$ a lowest weight vector in its dual right representation such that 
$\la i | Id | i\ra = 1$.
The action of $\mathfrak{g}$ on $V_i$ is 
\begin{align}
X|i\rangle &= 0, \ \ \mbox{if} \ X \in \mathfrak{n}_+, \label{3.17} \\
 h_{\alpha_j}|i\rangle &= \delta_{ij}. \notag
\end{align}


From \er{decomp}, we have
$$
L = F \exp(-H) M^{-1}.
$$
Therefore using the ${}^*$ operation in \er{dual}, we have 
\begin{multline}\label{nice}
\la i | L^* L | i\ra = \la i | (M^{-1})^* \exp(-\bar H) F^* F \exp(-H) M^{-1}| i \ra\\
 = \Big\la i \Big| \exp\Big(-\sum_{j=1}^m (b_j + \bar b_j) h_{\a_j}\Big) \Big| i \Big\ra = e^{-(b_i + \bar b_i)},
\end{multline}
where we have used the following facts. 
First, by $F\in K$ we have $F^* F = Id$. 
Secondly, since $M^{-1}\in N_+$ by \eqref{3.17} we have $M^{-1}|i\ra = |i\ra$.  
Similarly, $(M^{-1})^*\in N_-$ and $\la i |$ is a lowest weight vector, so $\la i | (M^{-1})^* = \la i |$.

Eq. \er{nice} actually shows that $\la i | L^* L | i\ra$ is real.
Therefore by \er{bu}, 
\begin{equation}\label{form}
U_i = -\log \la i | L^* L | i \ra + p_i.
\end{equation}

Now we show that for the above $U_i$ to satisfy \er{toda2} with $L$ from \er{Lcomb}, we must have 
$$
p_i = \sum_{j=1}^n a^{ij} \log |f_j(z)|^2.
$$
This follows from \cite{LSbook}*{\S 4.1.2} using the so-called Jacobi identity from \cite{LSbook}*{\S 1.6.4}.
The identity says that for a general element $g\in G^s$, the simply-connected Lie group with Lie algebra $\fg$,
we have 
\begin{equation}\label{Jacobi}
\la i | g | i \ra \la i | e_{\a_i} g e_{-\a_i} | i \ra - \la i | g e_{-\a_i} | i \ra \la i | e_{\a_i} g | i \ra = \prod_{j\neq i} \la j | g | j \ra^{-a_{ij}}.
\end{equation}

From \er{form} and that $p_i$ is harmonic, we have 
\begin{equation}\label{quot rule}
U_{i, z\bar z} = - \frac{\la i | L^* L | i\ra\la i | L^* L | i\ra_{z\bar z} - \la i | L^* L | i\ra_z\la i | L^* L | i\ra_{\bar z}}{\la i | L^* L | i\ra^2}.
\end{equation}
Now by \er{L hol} and \er{Lcomb}, we have
$$
\la i | L^* L | i\ra_z = \la i | L^* L\eta | i\ra = f_i(z) \la i | L^* L e_{-\a_i} | i\ra,
$$
where we have used that for the $i$th fundamental representation we have $e_{-\a_j} | i \ra =0$ for $j\neq i$ (see Eq. \er{spell action}). 
Taking the ${}^*$ operation and noting \er{dual}, \er{Lcomb} also gives
\begin{equation*}\label{L*}
(L^*)_{\bar z} (L^*)^{-1} = \eta^* = \sum_{i=1}^n \ol{f_i(z)} e_{\a_i}.
\end{equation*}
Therefore, similarly we have 
$$
\la i | L^* L | i\ra_{\bar z} = \la i | \eta^* L^* L | i\ra = \ol{f_i(z)} \la i | e_{\a_i}L^* L | i\ra.
$$
Furthermore, we have 
$$
\la i | L^* L | i\ra_{z\bar z} = |{f_i}|^2 \la i | e_{\a_i}L^* L e_{-\a_i} | i\ra.
$$
Now applying the Jacobi identity \er{Jacobi} to \er{quot rule} with $g=L^* L$ and using $a_{ii} = 2$ we get
$$
U_{i, z\bar z} = - |f_i|^2 \prod_{j=1}^n \la j | L^* L | j\ra^{-a_{ij}}
$$
 By \er{form}, this is 
\begin{align*}
U_{i, z\bar z} &= - |f_i|^2 \exp\Big(\sum_{j=1}^n a_{ij} U_j - \sum_{j=1}^n a_{ij} p_j\Big) \\
&= - \exp\Big(\log |f_i|^2 - \sum_{j=1}^n a_{ij} p_j\Big) \exp\Big({\sum_{j=1}^n a_{ij} U_j}\Big).
\end{align*}
Therefore for the $U_i$ to satisfy \er{toda2}, we need $\log |f_i|^2 - \sum_{j=1}^n a_{ij} p_j=0$. This gives 
$$
p_i = \sum_{j=1}^n a^{ij} \log |f_j|^2
$$
and proves the formula \er{loc sol}. 
\end{proof}
\begin{remark}
A holomorphic map $\eta  : D\to \fg_{-1}$ as in \er{the eta} in a simply-connected domain $D$
also gives rise to a solutions $U= (U_1, U_2, \cdots, U_n)$ to the Toda system \eqref{toda2} in $D$.  

Indeed, if $ \eta(z)= \sum_{i=1}^n f_i(z) e_{-\a_i}, z\in D,$ with $f_i(z) \neq 0$ in $D$ for all $1\leq i\leq n,$
we can construct  a $L: D\to G$ by solving the ODE: 
$
L^{-1}(z)L_z(z) = \eta(z), \ z \in D.
$
Here $L(z_0)$ need not have to satisfy $L(z_0) = Id.$
Now construct  $U_i$ as in \er{loc sol}. Then we can easily check that $U_i, 1\leq i\leq n,$ satisfies the Toda system \er{toda2} in the same way as above, where again the important point is the Jacobi identity \er{Jacobi}.
\end{remark}
\begin{corollary} \label{realana}
The components $U_i$ in any solution $U=(U_1,\dots,U_n)$ of the Toda system \eqref{toda2} are real analytic in $\mathbb{C}^*$.
\end{corollary}
\begin{proof}
Since the $f_i$ and $L$ in \eqref{loc sol} are holomorphic and the $f_i$ are nowhere zero in a neighbourhood of $z_0,$ it follows that the RHS
of \eqref{loc sol} is real analytic in a neighbourhood of $z_0 \in \mathbb{C}^*.$ Therefore for any $z_0 \in \mathbb{C}^{*},$ by Theorem \ref{Taiwan}, $U_i$ is real analytic in a neighbourhood of $z_0.$
This proves the corollary.
\end{proof}

From now on we always assume $D$ in Theorem \ref{Taiwan} is a neighbourhood of $1$
in $\mathbb{C}\backslash \mathbb{R}_{\leq 0}.$

\section{$W$-Invariants of the Toda systems}

In this section, we will introduce $W$-invariants of Toda systems. Then we present a result relating the $W$-invariants with the local solutions \eqref{loc sol} from the last section. 

By definition, a $W$-invariant (also called a characteristic integral)
	for the Toda system \er{toda2} 
	is a polynomial in the $\p_z^k U_i$ for $k\geq 1$ and $1\leq i\leq n$ whose derivative with respect to $\bar z$ is
	zero if $U=(U_1,\dots,U_n)$ is a solution.

For the Liouville equation 
 the $W$-invariant is
	\begin{equation}\label{lw} 
	W = U_{zz} - U_z^2,
	\end{equation}
	and it is meromorphic since $W_{\bar z}=0$ for a solution $U$. 
	Furthermore, plugging in the local solution $U(z) = \log \left[|f^{\prime}(z)|/(1 + |f(z)|^2)\right]$, where $f$ is  holomorphic and the derivative $f'$ is nowhere zero,  
we have 
	\begin{equation}\label{Schw}
	W = \frac{1}{2} \bigg(\frac{f'''}{f'} - \frac{3}{2}\Big(\frac{{f''}}{{f'}}\Big)^2\bigg),
	\end{equation}
	that is, the $W$-invariant of the local solution becomes one half of the Schwarzian derivative of the developing map $f$ of the solution $U.$ 
	We aim to generalize such results to general Toda systems in this section. 
	
	For a general Toda system associated to a simple Lie algebra of rank $n$, there are $n$ basic $W$-invariants (see \cite{FF}) so that the other $W$-invariants are differential polynomials in these. For Lie algebras of type of $A_n, B_n, C_n$ and $G_2$, more concrete expressions of W-invariant similar to \eqref{lw} have been obtained in \cite{LWY} and \cite{jnmp}. In this section following
\cite{jnmp}, we will obtain the $W$-invariants for a general simple Lie algebra by applying the Drinfeld-Sokolov \cite{DS} gauge  directly to the solution $U_i.$


Let $\phi(z) = (\phi_1(z), \phi_2(z), \cdots, \phi_n(z))$ be a smooth (not necessarily holomorphic) function defined on a domain $\Omega \subset \mathbb{C}.$ A  differential polynomial in $\phi_i$  is a polynomial in the $\phi_i, 1\leq i \leq n,$ and their derivatives $\partial^k_z \phi_i$ of finite orders.  We will denote a differential polynomial in the $\phi_i$ by $S([\phi_1], [\phi_2], \cdots, [\phi_n])$ or simply by $S(\phi).$
		
	To obtain the $W$-invariants, we conjugate \er{zero-curv} by $\exp\Big(\frac{1}{2}\sum_{i=1}^n U_i h_{\a_i}\Big)$ to arrive at  
	another zero-curvature representation of the Toda system \er{toda2}: 
	$$
	\Big[\p_z + \epsilon - \sum_{i=1}^n U_{i, z} h_{\a_i}, \p_{\bar z} - \sum_{i=1}^n e^{u_i} e_{\a_i}\Big] = 0,
	$$
	where $\e=\sum_{i=1}^n e_{-\a_i}\in \fg_{-1}$ and $u_i = \sum_{j=1}^n a_{ij} U_j$. 
	(The current  zero-curvature equation has different signs from the version in \cite{jnmp}, but all results there continue to hold for the current version.)
	
	Let $\fs$ be a Kostant slice of $\fg$, that is, a homogeneous subspace $\fs$ with respect to the principal grading \er{grading} such that 
	\begin{equation}\label{Kos s}
	\fg = [\epsilon, \fg] \oplus \fs.
	\end{equation}
	Then it is known \cite{K2} that $\fs\subset \fn_+$ in \er{nn} and $\dim \fs=n$. 
	Let $\{s_j\}_{j=1}^n$ be a homogeneous basis of $\fs$ ordered with nondecreasing gradings from \er{grading}. 	
	
	We first state the following lemma about the existence and uniqueness of the so-called Drinfeld-Sokolov gauge. 
	\begin{lemma}\label{DS lemma} For $1\leq i\leq n$, let $\phi_i$ be smooth functions on a domain $\Omega \subset \bc$. 
	\begin{enumerate}
	\item There exists a unique map $\widetilde M:\Omega \to N_+$ and smooth functions $W_j(\phi), 1\leq j \leq n,$ such that 
	\begin{equation}\label{gen DS}
		\widetilde M\Big( \p_z + \epsilon - \sum_{i=1}^n \phi_i\, h_{\a_i} \Big) \widetilde M^{-1} = \p_z + \e + 
\sum_{j=1}^n W_j(\phi)s_j.
	\end{equation}
Furthermore, there is a differential polynomial $S_j$ depending only on the slice basis of $\mathfrak{g}$ such that 
\begin{equation*}
W_j(\phi) = S_j([\phi_1], [\phi_2], \cdots, [\phi_n]) = S_j(\phi).
\end{equation*}
	\item If the $\phi_i$ are holomorphic in $\Omega,$ then so is $\widetilde M.$

	\item If  $\phi_i = U_{i,z}$ where $U_i$ is a solution of \eqref{toda2} in $\Omega,$ then 
\begin{equation}\label{use Uz}
W_j = S_j([U_{1,z}], [U_{2,z}], \cdots , [U_{n,z}])
\end{equation}
is holomorphic.  
	\end{enumerate}
	\end{lemma}
	We refer the reader to \cite{DS}*{Proposition 6.1}  for a proof of (i) and to \cite{jnmp} for a proof
of (iii). Furthermore, (ii) follows directly from the construction for (i). We also refer to \cite{Feher2} for the application of the Drinfeld-Sokolov gauge to the $W$-algebras.

	
	From now on, we fix a solution $U= (U_1,U_2, \cdots, U_n)$ of \eqref{toda2}, and let $W_j$ be given by \eqref{use Uz} for 
$z \in \mathbb{C}^*.$ The set of $W_j$ is called the $W$-invariants of the solution $U.$	
	Here is the main result in this section which relates the $W$-invariants with the holomorphic functions $f_i$ in $\eta$ from \er{the eta} for local solutions. 
	
\begin{theorem}\label{two-invs}
Let $U$ be a solution of \eqref{toda2} and the holomorphic data of $U(z), z \in D,$ be given as in  \eqref{the eta}.
Then there exists a unique map $M_1: \widetilde D \to N_+$, defined in a subdomain $\widetilde D\subset D$, such that 
	\begin{equation}\label{Peta}
	M_1 \Big(\p_z + \epsilon - \sum_{i=1}^n F_i h_{\a_i} \Big) M_1^{-1} = \p_z + \e + \sum_{j=1}^n W_j s_j \ \ z \in \widetilde D,
	\end{equation}
	where 
	\begin{equation}\label{def f}
	F_i = \sum_{j=1}^n a^{ij}\p_z \log f_j = \sum_{j=1}^n a^{ij} \frac{f_j'}{f_j}.
	\end{equation}

\end{theorem}

\begin{proof} Our proof follows the general approach in \cite{Feher0, Feher1, Feher2} of treating Toda theories as conformally reduced WZNW theories. For that purpose,  
we choose
\begin{equation}\label{choose Q}
Q_1 = \exp\Big( -\sum_{k=1}^n \big(\sum_{j=1}^n a^{kj} \log f_j\big) h_{\a_k} \Big): D\to \ch,
\end{equation}
where $\ch$ is the Cartan subgroup of $G$. 
Note that a single-valued branch of $\log f_j$ can be chosen since $D$ is simply-connected. 
Since $[h_{\a_k}, e_{-\a_i}] = -a_{ik} e_{-\a_i}$ by \er{cartan}, 
we have
\begin{equation} \label{compQ1}
\begin{split}
Q_1^{-1}e_{-\a_i}Q_1 &= \exp\Big( - \sum_{k=1}^n \big(\sum_{j=1}^n a^{kj} \log f_j\big) a_{ik} \Big) e_{-\a_i}  \\
&= \exp\big( - \log f_i \big) e_{-\a_i}= \frac{1}{f_i} e_{-\a_i}.
\end{split}
\end{equation}
It is also clear that 
\begin{equation} \label{compQ2}
Q_1^{-1}\p_z Q_1 = -\sum_{i=1}^n\big(\sum_{j=1}^n a^{ij} \p_z \log f_j\big) h_{\a_i} = -\sum_{i=1}^n F_i h_{\a_i}.
\end{equation}
Define 
\begin{equation*}
\Psi = Q_1^* L^* L Q_1.
\end{equation*}
Then by 
\er{the eta}, and the above, 
\begin{equation}\label{J curr}
\Psi^{-1} \Psi_z = (LQ_1)^{-1}(LQ_1)_z = Q_1^{-1} L^{-1} L_z Q_1 + Q_1^{-1}\p_z Q_1 = \epsilon -\sum_{i=1}^n F_i h_{\a_i}. 
\end{equation}
This quantity $\Psi^{-1} \Psi_z $ is called the current of $\Psi$ and is denoted by $J$ in the works \cites{Feher0, Feher1, Feher2}, and 
the requirement that its component in $\fn_-$ is $\epsilon$ is the key idea in these works. 

Furthermore, since $Q_1|i\ra = \exp(-\sum_{j=1}^n a^{ij}\log f_j) |i\ra$, the local solutions in \er{loc sol} have the following neat form
\begin{equation}\label{Feher-neater}
U_i = -\log\la i| Q_1^* L^* L Q_1|i\ra = -\log \la i | \Psi | i \ra.
\end{equation}

Now $\Psi$ has Gauss decomposition \er{Gdec} in a domain $\widetilde D \subset D$: 
\begin{equation}\label{new Gauss}
\Psi = \Psi_- \Psi_0 \Psi_+,\quad \Psi_{\pm} \in N_{\pm},  \Psi_0\in \ch.
\end{equation} 
Then clearly
$U_i = -\log \la i | \Psi_0 | i \ra$ in $\widetilde D.$ 
Therefore, we see that 
\begin{equation}\label{psi0}
\Psi_0 = \exp\Big(-\sum_{i=1}^n U_i h_{\a_i}\Big), \quad \Psi_0^{-1} \Psi_{0, z} = - \sum_{i=1}^n U_{i, z} h_{\a_i}.
\end{equation}

Plugging \er{new Gauss} in \er{J curr}, we get 
\begin{equation}\label{plugin}
\Psi_+^{-1}\Psi_{+, z} +  \Psi_+^{-1}\Psi_0^{-1}\Psi_{0, z} \Psi_{+} +  \Psi_+^{-1}\Psi_0^{-1}\Psi_-^{-1}\Psi_{-, z}\Psi_{0} \Psi_{+} = \epsilon -\sum_{i=1}^n F_i h_{\a_i}\ \ \mbox{in} \ \widetilde D.
\end{equation}
Comparing the components of the above equality in $\fn_-$, we see that
$$
\Psi_0^{-1}\Psi_-^{-1}\Psi_{-, z}\Psi_{0} = \epsilon.
$$
Also in view of \er{psi0}, Eq. \er{plugin} becomes
$$\Psi_+^{-1}\Psi_{+, z} +  \Psi_+^{-1}\Big(\epsilon - \sum_{i=1}^n U_{i, z}h_{\a_i}\Big) \Psi_{+} = \epsilon -\sum_{i=1}^n F_i h_{\a_i} \ \ \mbox{in} \ \widetilde D.$$
We owe this important equality to \cite{Feher2}*{Eq. (2.14a)}. 
Therefore, 
\begin{equation}\label{direct gauging}
\Psi_+ \Big(\p_z + \epsilon -\sum_{i=1}^n F_i h_{\a_i}\Big) \Psi_+^{-1} = \p_z + \epsilon - \sum_{i=1}^n U_{i, z}h_{\a_i} \ \ \mbox{in} \ \widetilde D.
\end{equation}

By Lemma \ref{DS lemma}, we see that \er{Peta} holds with $M_1 = \widetilde M_0 \Psi_+,$ where $\widetilde M_0 = \widetilde M|_{\widetilde D},\widetilde M : \mathbb{C} \backslash \{0\} \rightarrow N_{+}$ is the unique map satisfying \eqref{gen DS} with $\phi_i = U_{i,z}.$ Since the $F_i$ are holomorphic, we see that $M_1$ is holomorphic by Lemma \ref{DS lemma} (ii) although neither $\Psi_+$ nor $M_0$ is holomorphic. 
\end{proof}

\section{Applying the finite integral condition to $W$-invariants}

In this section, we adapt the analytical estimates from \cites{BM, LWY} to determine the simple forms of the $W$-invariants of our solutions to \eqref{toda2}. This explicit form is crucial to derive Theorem \ref{main} in the next section.

We define the weight of $\partial_z^kU_i$ to be $k.$
For a differential monomial in the $U_i$, we call by its \emph{degree} the sum of the weights multiplied by the algebraic degrees of the corresponding factors. 
	For example the above $W = U_{zz} - U_{z}^2$ in \er{lw} for the Liouville equation has a homogeneous degree $2$. 
	
	
\begin{proposition} The $W$-invariants for the Toda system \er{toda2} are 
	\begin{equation}\label{the W}
	W_j = \frac{w_j}{z^{d_j}}, \quad z\in \bc^*, \ 1\leq j\leq n, 
	\end{equation}
where the $d_j$ are the homogeneous degrees of the differential polynomials $S_j$ and the $w_j$ are constants. 
\end{proposition}

\begin{proof} This proof is an adaption of the proof in \cite{LWY} of the corresponding assertion in their Eq. (5.9). 

	Following \cite{LWY}*{Eq. (5.10)}, introduce 
\begin{equation}\label{def V}
V_i = U_i - 2\gamma^i \log |z|, \quad 1\leq i\leq n.
\end{equation}
Then system \er{toda2} becomes
$$
\begin{cases}
\displaystyle\Delta V_i = -4 |z|^{2\gamma_i} \exp\Big(\sum_{j=1}^n a_{ij} V_j\Big),\\
\displaystyle \int_{\br^2} |z|^{2\gamma_i} \exp\Big(\sum_{j=1}^n a_{ij} V_j\Big) \, dx <\infty.
\end{cases}
$$
As $\gamma_i > -1$, applying Brezis-Merle's argument in \cite{BM}, we have that $V_i \in C^{0, \alpha}$ on $\bc$ for some $\alpha \in (0, 1)$ and that they are upper bounded over $\bc$. 
Therefore we can express $V_i$ by the integral representation formula, and we have for $k\geq 1$ 
\begin{equation}
\label{order of V}
\begin{split}
\p_z^k V_i(z) &= O(1+|z|^{2+2\gamma_i -k}) \quad\, \text{ near } 0,\\
\p_z^k V_i(z) &= O(|z|^{-k}) \quad\quad\quad\qquad\text{ near } \infty.
\end{split}
\end{equation}

Therefore from \er{def V}, we have for $k\geq 1$
\begin{equation}
\label{order of U}
\begin{split}
\p_z^k U_i(z) &= O(|z|^{-k}) \quad\text{ near } 0,\\
\p_z^k U_i(z) &= O(|z|^{-k}) \quad\text{ near } \infty.
\end{split}
\end{equation}
	By $W_{j, \bar z} = 0$ and that $W_j$ has degree $d_j$, we see from the above estimates that $z^{d_j} W_j$ is holomorphic and bounded on $\bc^*$. 
	Therefore $z^{d_j} W_j = w_j$ is a constant by the Liouville theorem, and so \er{the W} holds. 
\end{proof}
\begin{remark}
It is known from \cite{FF} and also clear from \er{use Uz} that the homogeneous degree $d_j$ of $S_j$ is the same as the degree of the corresponding primitive adjoint-invariant function
	of the Lie algebra $\fg$ \cite{K1}. We call such degrees the degrees of the simple Lie algebra and we have listed them in Appendix subsection \ref{sub-degree}. 
\end{remark}
\begin{theorem}\label{prop-use-zg} The $W$-invariants $W_j$ for the Toda system \er{toda2} are also computed by 
\begin{equation}\label{use zg}
M_2\Big( \p_z + \e - \sum_{i=1}^n \frac{\gamma^i}{z} h_{\a_i} \Big) M_2^{-1} = \p_z + \e + \sum_{j=1}^n W_j s_j,
\end{equation}
where $M_2: \bc^*
\to N_+$ is unique and holomorphic. 
\end{theorem}

\begin{proof} 	
	By \er{def V}, 
	$$
	U_{i, z} = V_{i, z} + \p_z (2\g^i\log|z|) = V_{i, z} + \frac{\g^i}{z}.
	$$
	From \er{the W}, $W_j=\frac{w_j}{z^{d_j}}$. 
	Since $\gamma_i> -1$, by \er{order of V} and \er{order of U} for the orders at 0, we see that all the terms involving $\p_z^k V_i$ will 
	not appear in the final form of $W_j$ since their orders of pole are not high enough. Therefore in terms of \er{use Uz}, we have
	\begin{equation}\label{rid}
	W_j = S_j([U_{1, z}], \cdots, [U_{n, z}]) = S_j\Big(\big[\frac{\g^1}{z}\big], \cdots, \big[\frac{\g^n}{z}\big]\Big), \quad 1\leq j\leq n,
	\end{equation}
	where we recall that the brackets indicate that the $S_j$ are differential polynomials and depend on the derivatives of the arguments. 
	
	By Lemma \ref{DS lemma}, there exists a unique holomorphic $M_2: \bc^*\to N_+$ such that 
	\begin{equation*}
	M_2\Big( \p_z + \epsilon - \sum_{i=1}^n \frac{\gamma^i}{z} h_{\a_i} \Big) M_2^{-1} = \p_z + \e + \sum_{j=1}^n \wt{W_j} s_j,
	\end{equation*}
	where
	$$
	\wt{W_j} = S_j\Big(\big[\frac{\g^1}{z}\big], \cdots, \big[\frac{\g^n}{z}\big]\Big) = W_j, \quad 1\leq j\leq n,
	$$
	by \er{rid}. 
	\end{proof}

\section{The holomorphic functions in the local solutions}

The $W$-invariants play essential roles in our approach of classifying the solutions. The work \cite{LWY} classified the solutions to the Toda systems of type $A_n$ by relating them to an ODE whose coefficients are the $W$-invariants. 
In this section, we will use the $W$-invariants to largely restrict the holomorphic functions $f_i(z)$ in the local solutions \er{loc sol} to be $f_i(z) = z^{\gamma_i}$ as long as we allow some constant group element. 

\begin{theorem} Let $U_i$ be the solution of \eqref{toda2}, then
\begin{equation}\label{use g}
U_i(z) = -\log \la i | \Phi^* g^* g \Phi | i \ra + {2\gamma^i}\log |z| , \quad 1\leq i\leq n, \ \ z \in \mathbb{C} \backslash \mathbb{R}_{\leq 0}.
\end{equation}
where $\Phi$ satisfies \er{develop0} and $g\in G$ is a constant group element. 

\end{theorem}

\begin{proof} 
Set
\begin{equation}\label{def q2}
Q_2 = \exp\Big( -\sum_{k=1}^n ({\gamma^k}\log z) h_{\a_k} \Big): \bc\backslash \br_{\leq 0}\to  \ch.
\end{equation}
Then similarly to \er{compQ1} and \er{compQ2}, we get using \er{develop0} 
\begin{equation}\label{phiq2}
(\Phi Q_2)^{-1} (\Phi Q_2)_z = Q_2^{-1} \Big(\sum_{i=1}^n z^{\gamma_i} e_{-\a_i}\Big) Q_2 + Q_2^{-1} Q_{2, z} = \e - \sum_{i=1}^n \frac{\gamma^i}{z} h_{\a_i}.
\end{equation}

Let $M= M_2^{-1} M_1 : D\to  N_+$, where $M_1$ and $M_2$ are from Theorems \ref{two-invs} and \ref{prop-use-zg}. 
Then 
$$
M\Big(\p_z + \e - \sum_{i=1}^n F_i h_{\a_i}\Big) M^{-1} = \p_z + \e - \sum_{i=1}^n \frac{\g^i}{z} h_{\a_i},
$$
which is
$$
\e - \sum_{i=1}^n F_i h_{\a_i}  =M^{-1}\Big(\e - \sum_{i=1}^n \frac{\g^i}{z} h_{\a_i}\Big) M + M^{-1} M_z .
$$
This and \er{phiq2} imply that 
\begin{align*}
(\Phi Q_2 M)^{-1} (\Phi Q_2 M)_z &= M^{-1} \Big(\e - \sum_{i=1}^n \frac{\gamma^i}{z} h_{\a_i}\Big) M + M^{-1} M_z \\
&= \e - \sum_{i=1}^n F_i h_{\a_i}
= (LQ_1)^{-1}(LQ_1)_z,
\end{align*}
by \er{J curr}. Thus $[(\Phi Q_2 M)(LQ_1)^{-1}]_z = 0.$ Since both $\Phi Q_2 M$ and $LQ_1$ are holomorphic,  we have
\begin{equation}\label{important}
LQ_1 = g\Phi Q_2 M,
\end{equation}
where $g\in G$ is a constant element. 

Therefore by \er{Feher-neater}, \er{important} and \er{def q2}, the solution $U_i$ of \er{toda2} becomes 
\begin{equation} \label{locexp}
\begin{split}
U_i(z) 	&= -\log \la i | Q_1^* L^* L Q_1 | i \ra 
	= -\log \la i | M^* Q_2^* \Phi^* g^* g \Phi Q_2 M | i \ra \\
	&= -\log \la i | \Phi^* g^* g \Phi | i \ra + 2\gamma^i \log|z|, \ \ z \in D.
\end{split}
\end{equation}
By Corollary \ref{realana}, $U_i$ is real analytic. Since both sides of \eqref{locexp} are real analytic functions defined on  $\mathbb{C} \backslash \mathbb{R}_{\leq 0},$ \eqref{locexp} holds in  $\mathbb{C} \backslash \mathbb{R}_{\leq 0}.$
\end{proof}


\noindent{\bf Proof of Theorem \ref{main}.}
\begin{proof}
We write the Iwasawa decomposition \er{Iwasawa} for the constant element $g\in G$ as 
$$
g = F\Lambda C
$$
with $F\in K$, $\Lambda \in A$ and $C\in N$. Then using  $F^*F=Id$ and $\Lambda^* = \Lambda$ the solution \er{use g} becomes \er{Ui}.  This completes the proof.
\end{proof}

\appendix

\section{Background on Lie algebras and representation theory}  
\subsection{Iwasawa decomposition} \label{Iwasawasec}
Let ${}^* :\fg \to \fg$ be the conjugate-linear transformation defined by 
\begin{equation}\label{dual} 
{}^*|_{\fh_0} = \text{Id}|_{\fh_0}, 
\quad e_\a^* = e_{-\a},\ \a\in \Delta.
\end{equation}

Let $G$ be a connected complex Lie group whose Lie algebra is $\fg$. 
Then we have the following Iwasawa decomposition \cite{Knapp}*{Theorem 6.46}  
\begin{equation}\label{Iwasawa}
G = KAN,
\end{equation}
where $K$ is a maximal compact subgroup, $A$ is an abelian subgroup corresponding to $\fh_0$, and $N$ is the nilpotent subgroup corresponding to $\fn=\fn_-$ . 
The subgroups $A$ and $N$ are simply-connected. An element $F$ belongs to $K$ if and only if $F^* F = Id$, where $*$ is lifted from $\fg$ to $G$. 

The following Cartan involution $\theta: \mathfrak{g} \rightarrow \mathfrak{g}$ defined by 
\begin{equation}\label{def theta}
\theta(X) =X^{\theta}= -X^*, \quad X \in 
\mathfrak{g} 
\end{equation}
is very 
important, since it satisfies $\theta([X,Y]) = [\theta(X), \theta(Y)].$ Then $K$ is characterized as the fixed point set $G^{\theta}$
in $G$ under the lifted action of $\theta$ and its Lie algebra is $\mathfrak{k} = \mathfrak{g}^{\theta}$  (see \cite{Knapp}*{Theorem 6.31}).

\subsection{Basic representation theory} \label{basicrepresentation}
The integral weight lattice of $\fg$ is $\Lambda_W = \{\beta\in \fh_0'\,|\, \beta(h_{\a_i})\in \bz,\ \forall 1\leq i\leq n\}$. An integral weight $\beta$ is called dominant if $\beta(h_{\a_i})\geq 0$ for all $1\leq i\leq n$. The weight lattice is a lattice generated by the fundamental weights $\omega_i$ for $1\leq i\leq n$ satisfying
$$
\omega_i(h_{\a_j}) = \delta_{ij}.
$$

An irreducible representation $\rho$ of $\fg$ on a finite-dimensional complex vector space $V$ has the weight space decomposition $V = \oplus V_\beta$, where $\beta\in \Lambda_W$ and $V_\beta = \{ v\in V \,|\, \rho(H) (v) = \beta(H) v,\ \forall H\in \fh\}$. We have 
\begin{equation}\label{should list}
\rho(\fg_\a) V_\beta \subset V_{\a + \beta}.
\end{equation} 
A basic theorem states that if $\rho$ is an irreducible representation, then there exists a unique highest weight $\lambda$ with a one-dimensional highest weight space $V_\lambda$ such that $\rho(\fn_+) V_\lambda = 0$. All the weights of $V$ are of the form $\lambda - \sum_{i=1}^n m_i \a_i,$ where the $m_i$ are nonnegative integers. 

The Theorem of the Highest Weight \cite{Knapp}*{Theorem 5.5} asserts that up to equivalence, the irreducible finite-dimensional complex representations of $\fg$ stand in one-one correspondence with the dominant integral weights which sends an irreducible representation to its 
highest weight. We denote the irreducible representation space corresponding to a dominant weight $\lambda$ by $V^\lambda$. 

There is a canonical pairing between the dual space $V^*=\on{Hom}(V, \bc)$ and $V$ denoted by $\la w, v \ra \in \bc$ with $v\in V$ and $w\in V^*$. $V^*$ has a \emph{right representation} $\rho^*$ of $\fg$ defined by 
\begin{equation}\label{def right}
\la w \rho^*(X), v\ra = \la w, \rho(X)v\ra,\quad X\in \fg.
\end{equation}
 
The representation corresponding to the $i$th fundamental weight $\omega_i$ is called the $i$th fundamental representation  of $\fg$, which we denote by $V_i$. 
We choose a highest weight vector in $V_i$, and following the physicists \cite{LSbook} we called it by $|i\ra$. 
We choose a vector $\la i |$ in the lowest weight space in $V_i^*$ and require that $\la i | Id | i\ra = 1$ for the identity element $Id\in G$.  
For simplicity, we will omit the notation $\rho$ for the representation. 

We have (see \cite{LSbook}*{Eq. (1.4.19)})
\begin{equation}\label{spell action}
X|i\ra = 0, \ \forall X\in \fn_+;\quad h_{\a_j}|i\ra = \delta_{ij} |i\ra;\quad \text{and }e_{-\a_j}|i\ra = 0,\ \forall j\neq i.
\end{equation} 
That is, in the $i$th fundamental representation, only $e_{-\a_i}$ may bring the highest weight vector down. Similarly we have $\la i | Y = 0$ for $Y\in \fn_-$, and $\la i | e_{\a_j} = 0$ if $j\neq i$. 

Let $G^s$ be the simply connected Lie group with Lie algebra $\mathfrak{g}.$ Let $V$ be a finite dimensional, irreducible representation of $\mathfrak{g}.$ To lift the representation to $G^s,$ first we need to define the representation in a neighborhood of the identity. Since
in a neighborhood of the identity any element of $G^s$ can be uniquely written as $e^X$ for some $X \in \mathfrak{g},$ we can
simply define
\begin{align*}
e^X v = \exp(X)v, \ \ v \in V,
\end{align*}
where $\exp(X) = \sum_{k=0}^{\infty} \frac{X^k}{k !}.$ Then using the simply connectedness of $G^s,$ extend the map to any point of $G^s$
by using a path joining the identity and the point. Then one can show that the definition is independent of the choice of the path and the map defined this way is indeed a representation of $G^s.$ 



The universal enveloping algebra of $\fg$ is defined as 
$$U(\fg) = T(\fg)/J,$$ 
where $T(\fg)=\oplus_{k=0}^\infty T^k(\fg)$ is the tensor algebra, and $J$ is the two-sided ideal generated by all elements of the form $X\otimes Y - Y\otimes X - [X, Y]$ with $X$ and $Y$ in $\fg$. 
A representation of $\fg$ also leads to a representation for the universal enveloping algebra $U(\fg)$ \cite{Knapp}. Similarly, we can define the universal enveloping algebra of a subalgebra of $\fg$, for example $U(\fn_-)$ for $\fn_-$. 

For $\mu, \nu\in U(\fg)$ and $g \in G^s$, $\la i |\nu g \mu| i \ra$ denotes the pairing of $\la i |\nu$ in $V_i^*$ with $g(\mu |i\ra)$ in $V_i$.

In our main Theorem \ref{main}, we can work with a general Lie group $G$ whose Lie algebra is $\fg$ instead of only the simply-connected $G^s$. The reason is that 
the simply-connected compact subgroup $K^s$ of $G^s$ is used only in passing. 
Our results are expressed using $N$ and $A$, and they are simply-connected and the same for a general $G$ and for the simply-connected $G^s$. 

There is a more concrete realization of the dual $V_i^*$ in  \ref{def right}. By the unitary trick, there exists a Hermitian form $\{\cdot, \cdot\}$ on $V_i$ (conjugate linear in the second position) invariant under the compact group $K^s$ of a simply-connected $G^s$. 
The important property of this Hermitian form is that \cite{KosToda}*{Eq. (5.11)} 
\begin{equation}\label{adjoint}
\{ g u , v \} = \{u, g^* v\}, \quad g\in G^s, \ u, v\in V_i. 
\end{equation}


Choose $v^{\omega_i}\in V_i$ to be a highest weight vector for the $i$th fundamental representation, and we require that $\{v^{\omega_i}, v^{\omega_i}\}=1$. 
Then the term in \er{Ui} is, with $g=\Lambda C\Phi$,  
\begin{equation}\label{use vi}
\la i | g^* g | i \ra = \{g^* g v^{\omega_i}, v^{\omega_i}\} = \{g v^{\omega_i}, g v^{\omega_i}\} > 0.
\end{equation}

The Weyl group $W$ of a Lie algebra $\fg$ is the finite group generated by the reflections in the simple roots on $\fh'_0$ 
$$
s_i (\beta) = \beta - \beta(h_{\a_i}) \a_i, \quad 1\leq i\leq n.
$$
In the Weyl group, there is a unique element $\kappa\in W$ that is the longest element in the sense that when one writes it as a product of the simple reflections it has the maximal length. 


The weights of the irreducible representation $V^\lambda$ with a highest weight $\lambda$ are invariant under the Weyl group, and its lowest weight is $\kappa \lambda$ \cite{KosToda}*{Eq. (5.2.10)}.

\subsection{Principal grading and degrees of primitive adjoint-invariant functions}\label{sub-degree}
		Using $E_j$ from \er{propertyejw0}, define the so-called principal grading element 
	\begin{equation}\label{E0}
	E_0 = \sum_{j=1}^n E_j \in \fh_0,\quad\text{such that  }\a_i(E_0) = 1,\ \text{for }1\leq i\leq n.
	\end{equation}
	Define $\fg_j = \{ x\in \fg\,|\, [E_0, x]= j x\}$. Then 
	\begin{equation}\label{grading}
	\fg = \bigoplus_j \fg_j
	\end{equation}
	is the principal grading of $\fg$.
	
The degrees of the primitive homogeneous adjoint-invariant functions of the simple Lie algebras are listed as follows
\medskip
\begin{center}
\allowdisplaybreaks{
\begin{tabular}{cc}
\hline
Lie algebras &  degrees \\
\hline
$A_n$ & $2, 3, \cdots, n+1$ \\
$B_n$ & $2, 4, \cdots, 2n$ \\
$C_n$ & $2, 4, \cdots, 2n$ \\
$D_n$ & $2, 4, \cdots, 2n-2, n$ \\
$G_2$ & $2, 6$ \\
$F_4$ & $2, 6, 8, 12$ \\
$E_6$ & $2, 5, 6, 8, 9, 12$ \\
$E_7$ & $2, 6, 8, 10, 12, 14, 18$ \\
$E_8$ & $2, 8, 12, 14, 18, 20, 24, 30$ \\
\hline
\end{tabular}
}
\end{center}

\section{Cartan matrices for simple Lie algebras and their inverse matrices}
There are four infinite series of classical complex simple Lie algebras and five exceptional Lie algebras with the following Cartan matrices
\allowdisplaybreaks{
\begin{align*}
\label{Cartan-A}
&A_n =\fs\fl_{n+1}: 
\left(\begin{smallmatrix}
2 & -1 & & &\\
-1 & 2 & -1 & & \\
 & \ddots & \ddots & \ddots& \\
 & &-1 & 2 & -1\\
 & & & -1 & 2
 \end{smallmatrix}\right), 
&&B_n=\fs\fo_{2n+1} : \left(\begin{smallmatrix}
2 & -1 & & &\\
-1 & 2 & -1 & & \\
 & \ddots & \ddots & \ddots& \\
 & &-1 & 2 & -2\\
 & & & -1 & 2
 \end{smallmatrix}\right),\\
 &C_n = \fs\fp_{2n}: \left(\begin{smallmatrix}
2 & -1 & & &\\
-1 & 2 & -1 & & \\
 & \ddots & \ddots & \ddots& \\
 & &-1 & 2 & -1\\
 & & & -2 & 2
 \end{smallmatrix}\right),\ 
&& D_n = \fs\fo_{2n}: \left(\begin{smallmatrix}
2 & -1 & & & & \\
-1 & 2 & -1 & & & \\
 & \ddots & \ddots & \ddots& & \\
 & &-1 & 2 & -1 & -1\\
 & & & -1 & 2 & \\
 & & & -1 &  & 2\\
 \end{smallmatrix}\right),\\
 &G_2: \left(\begin{matrix}
 2 & -1\\
 -3 & 2
 \end{matrix}\right),
 && F_4: \left(\begin{smallmatrix}
 2 & -1 & & \\
 -1 & 2 & -2 & \\
  & -1 & 2 & -1\\
  & & -1 & 2
\end{smallmatrix}\right),\\
&E_6: \left(\begin{smallmatrix}
2 & & -1 & & & \\
 & 2 & & -1 & & \\
 -1 & & 2 & -1 & & \\
 & -1 & -1 & 2 & -1 & \\
 & & & -1 & 2 & -1 \\
 & & & & -1 & 2 
\end{smallmatrix}\right),
&&E_7: \left(\begin{smallmatrix}
2 & & -1 & & & & \\
 & 2 & & -1 & & & \\
 -1 & & 2 & -1 & & & \\
 & -1 & -1 & 2 & -1 & & \\
 & & & -1 & 2 & -1 & \\
 & & & & -1 & 2 & -1\\
 & & & & & -1 & 2
\end{smallmatrix}\right),\\
&E8: \left(\begin{smallmatrix}
2 & & -1 & & & & & \\
 & 2 & & -1 & & & & \\
 -1 & & 2 & -1 & & & & \\
 & -1 & -1 & 2 & -1 & & & \\
 & & & -1 & 2 & -1 & & \\
 & & & & -1 & 2 & -1 & \\
 & & & & & -1 & 2 & -1\\
 & & & & & & -1 & 2
\end{smallmatrix}\right).&&
\end{align*}
}In the above, the labelling of the roots for the exceptional Lie algebras follows \cite{Knapp}*{pp 180-1}. 
We require that $n\geq 2$ for $B_n$ and $C_n$, and $n\geq 3$ for $D_n$. Furthermore, we have the following isomorphisms
$$
B_2\cong C_2, \quad D_3 \cong A_3.
$$

The inverse of the Cartan matrices have the following form:
\begin{align*}
&A_n:
\left( {\begin{smallmatrix}
\frac{n}{n+1}& \frac{n-1}{n+1} &\frac{n-2}{n+1} &\cdots &\frac{2}{n+1} &\frac{1}{n+1}   \\
\frac{n-1}{n+1} &\frac{2n-2}{n+1} &\frac{2n-4}{n+1} &\cdots &\frac{4}{n+1} &\frac{2}{n+1}   \\
\frac{n-2}{n+1} &\frac{2n-4}{n+1} &\frac{3n-6}{n+1}  &\cdots &\frac{6}{n+1} &\frac{3}{n+1}   \\
\vdots &\vdots &\vdots&\ddots & \vdots &\vdots\\
\frac{2}{n+1}&\frac{4}{n+1}&\frac{6}{n+1}&\cdots&\frac{2n-2}{n+1}&\frac{n-1}{n+1}   \\
\frac{1}{n+1}&\frac{2}{n+1}&\frac{3}{n+1}&\cdots&\frac{n-1}{n+1}&\frac{n}{n+1}   \\
\end{smallmatrix} } \right),
&&B_n :
\left( {\begin{smallmatrix}
1&1&1&\cdots &1 &1\\
1&2&2&\cdots&2&2\\
1&2&3&\cdots&3&3\\
\vdots &\vdots &\vdots&\ddots & \vdots &\vdots\\
1&2&3&\cdots&n-1 &n-1\\
\frac{1}{2}&1&\frac{3}{2}&\cdots&\frac{n-1}{2}&\frac{n}{2}\\
\end{smallmatrix} } \right), \\
& C_n:
\left( {\begin{smallmatrix}
1&1&1&\cdots &1 &\frac{1}{2}\\
1&2&2&\cdots&2&1\\
1&2&3&\cdots&3&\frac{3}{2}\\
\vdots &\vdots &\vdots&\ddots & \vdots &\vdots\\
1&2&3&\cdots&n-1 &\frac{n-1}{2}\\
1&2&3&\cdots&n-1 &\frac{n}{2}\\
\end{smallmatrix} } \right), 
  && D_n:
  \left( {\begin{smallmatrix}
   1 & 1 & 1 & \cdots & 1 & \frac{1}{2} & \frac{1}{2}\\
    1 & 2 & 2& \cdots & 2 & 1 & 1 \\
 1 & 2 & 3 & \cdots & 3 & \frac{3}{2} & \frac{3}{2}\\
\vdots &\vdots &\vdots&\ddots & \vdots &\vdots &\vdots\\
   1 & 2 & 3 & \cdots & n-2 & \frac{n-2}{2} & \frac{n-2}{2}\\
\frac{1}{2} & 1 & \frac{3}{2} & \cdots & \frac{n-2}{2} & \frac{n}{4} & \frac{n-2}{4}\\
\frac{1}{2} & 1 & \frac{3}{2} & \cdots & \frac{n-2}{2} & \frac{n-2}{4} & \frac{n}{4}\\
  \end{smallmatrix} } \right), \\
&G_2:
 \left( {\begin{smallmatrix}
 2&1\\
3&2\\
  \end{smallmatrix} } \right), 
&&F_4:
 \left( {\begin{smallmatrix}
  2 &3 &4 &2 \\
3 &6 & 8& 4\\
2 &4&6&3 \\
1&2&3&2
  \end{smallmatrix} } \right), \\
&E_6 :
   \left( {\begin{smallmatrix}
   \frac{4}{3} & 1 & \frac{5}{3} & 2 & \frac{4}{3} & \frac{2}{3}\\
    1 & 2 & 2& 3 & 2 & 1  \\
 \frac{5}{3} & 2 & \frac{10}{3} & 4& \frac{8}{3} & \frac{4}{3}\\
   2 & 3 & 4 & 6 & 4 & 2 \\
\frac{4}{3} & 2 & \frac{8}{3} & 4 & \frac{10}{3}  & \frac{5}{3}\\
\frac{2}{3} & 1 & \frac{4}{3} & 2 & \frac{5}{3} & \frac{4}{3} \\
  \end{smallmatrix} } \right),
&&E_7:
 \left( {\begin{smallmatrix}
 2 &2 &3 &4 &3 &2 &1 \\
2& \frac{7}{2} & 4&6& \frac{9}{2}&3& \frac{3}{2} \\
3&4&6&8&6&4&2 \\
4&6&8&12&9&6&3\\
3&\frac{9}{2}&6&9&\frac{15}{2}&5&\frac{5}{2}\\
2&3&4&6&5&4&2\\
1&\frac{3}{2}&2&3&\frac{5}{2}&2&\frac{3}{2}\\
  \end{smallmatrix} } \right), \\
&E_8:
\left( {\begin{smallmatrix}
4&5&7&10&8&6&4&2\\
5&8&10&15&12&9&6&3\\
7&10&14&20&16&12&8&4\\
10&15&20&30&24&18&12&6\\
8&12&16&24&20&15&10&5\\
6&9&12&18&15&12&8&4\\
4&6&8&12&10&8&6&3\\
2&3&4&6&5&4&3&2\\
  \end{smallmatrix} } \right). 
\end{align*}

\bigskip
\end{document}